\documentclass{amsart}
\usepackage{amsmath, mathrsfs,amsthm, amssymb,color,fullpage, amsrefs,mathtools,marvosym,enumerate,cancel}
\usepackage[utf8]{inputenc}
\usepackage{hyperref}
\usepackage{tikz}
\usetikzlibrary{decorations.pathreplacing}
 \def\beq{\begin{eqnarray}}
\def\eeq{\end{eqnarray}}

\newtheorem{thm}{Theorem}[section]
\newtheorem*{thm*}{Theorem}
\newtheorem{prop}[thm]{Proposition}
\newtheorem*{prop*}{Proposition}
\newtheorem{cor}[thm]{Corollary}

\newtheorem{lemma}[thm]{Lemma}

\newtheorem{exa}[thm]{Example}

\newtheorem*{question*}{Question}

\newtheorem{defn}[thm]{Definition}
\newtheorem{remark}[thm]{Remark}

\newtheorem{result}[thm]{Result}
\newtheorem{Results}[thm]{Results}
\newtheorem{notarem}[thm]{Notation and Remarks}

\newtheorem{notaconv}[thm]{Notational Conventions}

\definecolor{pink}{rgb}{1,0,1}

    \newcommand{\R}{\mathbb R}
    \newcommand{\pa}{\partial}

     \newcommand{\y}{\mathcal Y}
          \newcommand{\bell}{\pmb{\ell}}
         \newcommand{\balpha}{\pmb{\alpha}} 
    
    \def\stek{\operatorname{Stek}}
\def\isos{\operatorname{Iso}_{\stek}}

    \newcommand{\cE}{\mathcal{E}}

    \newcommand{\Z}{\mathbb Z}
\newcommand{\Om}{\Omega}
\newcommand{\pmc}{\pmb{C}}
        
         \def\abs{\operatorname{ab}}  
          \def\pmca{\pmb{C}_{\abs}}
          \newcommand{\cp}{\mathcal{P}^*}
   \newcommand{\cpn}{\mathcal{P}^*(n)}
    \newcommand{\bxi}{\pmb{\xi}}
     \def\omred{\Om^{\operatorname{red}}}
     \newcommand{\cpp}{\mathcal{P}^{**}}

 \newcommand{\ub}{\mathbf{u}}

            \newcommand{\bgamma}{\pmb{\gamma}} 
          \newcommand{\bdelta}{\pmb{\delta}}
             \def\poly{\operatorname{poly}}
\def\isop{\operatorname{Iso}_{\poly}}

\title[Steklov spectral finiteness]{The Steklov spectrum of convex polygonal domains I:  spectral finiteness}

\author[E. Dryden]{Emily B. Dryden}
\author[C. Gordon]{Carolyn Gordon} 
\author[J. Moreno]{Javier Moreno}
\author[J. Rowlett]{Julie Rowlett}
\author[C. Villegas-Blas]{Carlos Villegas-Blas}

\address{Emily B. Dryden, Department of Mathematics,  Bucknell University,  Lewisburg, PA 17837 USA} 
\urladdr{\href{http://www.unix.bucknell.edu/~ed012/}{http://www.unix.bucknell.edu/~ed012/}}
\email{\href{mailto:}{emily.dryden@bucknell.edu}}

\address{Carolyn Gordon, Department of Mathematics, Dartmouth College, Hanover, NH 03755 USA} 
\urladdr{\href{https://math.dartmouth.edu/~gordon/}{https://math.dartmouth.edu/~gordon/}}
\email{\href{mailto:}{carolyn.s.gordon@dartmouth.edu}}

\address{Javier Moreno, Department of Mathematics,  Universidad de Los Andes,  111711, Bogotá, Colombia}
\email{\href{mailto:}{jd.morenop@uniandes.edu.co}}

\address{Julie Rowlett,  Mathematical Sciences, Chalmers University,  412 96, Gothenburg, Sweden} 
\urladdr{\href{http://www.math.chalmers.se/~rowlett}{http://www.math.chalmers.se/~rowlett}}
\email{\href{mailto:julie.rowlett@chalmers.se}{julie.rowlett@chalmers.se}}

\address{Carlos Villegas Blas, Instituto de Matemáticas, Unidad Cuernavaca, 
Universidad Nacional Autónoma de México, 62210, Cuernavaca, Morelos, Mexico}
\urladdr{\href{https://www.matem.unam.mx/fsd/villegas}{https://www.matem.unam.mx/fsd/villegas}}
\email{\href{mailto:carlos.villegas@im.unam.mx}{carlos.villegas@im.unam.mx}}

\begin{document}
\maketitle 

\begin{abstract}
    We explore the Steklov eigenvalue problem on convex polygons, focusing mainly on the inverse Steklov problem.  Our primary finding reveals that, for almost all convex polygonal domains, there exist at most finitely many non-congruent domains with the same Steklov spectrum.  Moreover, we obtain explicit upper bounds for the maximum number of mutually Steklov isospectral non-congruent polygonal domains.  Along the way, we obtain isoperimetric bounds for the Steklov eigenvalues of a convex polygon in terms of the minimal interior angle of the polygon. 
\end{abstract}

\section{Introduction} \label{s:intro}

The Steklov eigenvalue problem on a bounded planar domain $\Om$, first introduced by Vladimir Andreevich Steklov in 1895, consists of finding all $\sigma\in\R$ for which there exists $0\neq u\in C^\infty(\Om)$ satisfying
\beq \Delta u = 0 \textrm{ in } \Omega, \quad \frac{\pa u}{\pa n} = \sigma u \, \, \textrm{ on } \pa \Omega \label{eq:steklov} \eeq 
where $\Delta$ is the Laplacian and $\frac{\pa}{\pa n}$ is the 
outward-pointing normal derivative.  
The Steklov spectrum, i.e., the collection of all such $\sigma$ repeated with multiplicity, 
 is discrete and satisfies 
\beq 0 = \sigma_0(\Omega) < \sigma_1 (\Omega) \leq \cdots \leq \sigma_m (\Omega) \leq \cdots \nearrow + \infty. \label{eq:def_stek_ev} \eeq 
Equivalently, the Steklov eigenvalues are those of the Dirichlet-to-Neumann operator, which maps the Dirichlet boundary values of harmonic functions on $\Om$ to their Neumann boundary values. 

 If the boundary $\pa \Omega$ is piecewise $C^1$, then the Steklov eigenvalues have Weyl asymptotics of the form given in \cite{agr06}: 
\beq \sigma_m = \frac{\pi m}{|\pa \Omega|} + o(m), \quad \textrm{ as } m \to \infty. \label{eq:weyl} \eeq
It follows that the perimeter is a Steklov spectral invariant for domains with piecewise $C^1$ boundaries.   

The Steklov eigenvalue problem lay mostly dormant for many years.    A breakthrough came in 1954 when Weinstock \cite{Wein} proved that the unit disk uniquely maximizes $\sigma_1(\Om)$ among all simply-connected planar domains of perimeter one.     Recent decades have seen tremendous interest in the Steklov problem, not only for planar domains but for more general compact Riemannian manifolds with boundary.  The very rich tapestry of results includes asymptotics, isoperimetric eigenvalue bounds, optimization of eigenvalues and a remarkable relationship to free boundary minimal surfaces in balls, inverse spectral results, numerical results, and much more.    See the surveys  \cite{gir_pol} and \cite{pre_survey} for exposition and many references in this very rapidly expanding area.  For historical background and physical implications, we refer to Kuznetsov et. al \cite{legacy}.

The impetus for the current paper arose from the powerful results of Levitin, Parnovski, Polterovich and Sher in \cite{lpps19} and in the subsequent article \cite{klpps21}, joint also with Krymski.   The focus of these papers is on simply-connected curvilinear $n$-gons $\Om$ with all interior angles lying in $(0,\pi)$.     They associate to each such $\Om$ a trigonometric polynomial $P_\Om$, referred to as the \emph{characteristic polynomial} of $\Om$.  The polynomial depends only on the edge lengths and angles of $\Om$.    In the former paper, they show that the roots of $P_\Om$ yield the Steklov spectral asymptotics of $\Om$ up to order $O(m^{-\epsilon})$ for some $\epsilon>0$.   In the latter, they show that the characteristic polynomial is a Steklov spectral invariant.   By applying this invariant, they show for generic curvilinear $n$-gons with angles in $(0,\pi)$ that the Steklov spectrum determines the edge lengths, and it moreover determines the angles up to countably many explicit possibilities.  The genericity conditions, referred to as \emph{admissibility}, consist of an incommensurability condition on the edge lengths together with the exclusion of angles of the form $\frac{\pi}{2m+1}$ with $m\in \Z^+$.  

Motivated by their results, we address the question of finite Steklov spectral determination of convex polygons.   First, however, we prove Steklov eigenvalue bounds for compact Riemannian surfaces with boundary and for triangles. Then we further develop our results to apply to non-convex polygonal domains. The eigenvalue bounds lead to an additional spectral invariant.
To address eigenvalue bounds that are independent of scaling, we adopt the commonly used normalization by the perimeter of the boundary, i.e., we consider $\sigma_k(\Om)L(\pa \Om)$.  We prove:

\begin{result}{\rm [See Theorem~\ref{th:awesome} for a more precise statement.]}
For each $n=3, 4, \dots$, there exist constants $C_n>0$ and $\delta_n$ such that if $\Om$ is any convex $n$-gon  with smallest angle $\alpha(\Om)<\delta_n$, then the Steklov eigenvalues of $\Om$ satisfy 
$$\sigma_k(\Om) L(\partial \Om) \leq C_n k^2\alpha(\Om), \quad \text{for all } k \geq 0.$$  

\end{result}

\begin{cor}\label{cor:intro}
A lower bound on the $k$th normalized eigenvalue yields a lower bound on all the interior angles of $\Om$.   In particular, there is a uniform lower bound on the angles of any collection of mutually Steklov isospectral convex $n$-gons.\end{cor}

 We then address spectral finiteness using the characteristic polynomial and in some cases also Corollary~\ref{cor:intro}:  

\smallskip

\begin{Results}~
\begin{itemize}
\item[(a)] {\rm [See Section~\ref{sec: isospec size}.]} For every convex $n$-gon $\Om$ that satisfies the generic conditions of admissibility, we obtain an explicit finite upper bound on the number (up to congruence) of convex $n$-gons with the same Steklov spectrum as $\Om$.   If, moreover, all angles of the admissible convex $n$-gon $\Om$ are obtuse, then $\Om$ is uniquely determined by its Steklov spectrum among all convex $n$-gons.

\item[(b)] {\rm [See Section~\ref{sec: finiteness}.]}   For convex $n$-gons satisfying some genericity conditions that are weaker than admissibility, we obtain further Steklov finiteness results by applying Corollary~\ref{cor:intro} along with the characteristic polynomial.

\end{itemize}
\end{Results}

We emphasize that additional tools need to be developed if one hopes to remove genericity assumptions completely.  Indeed, as noted in \cite{klpps21}, all parallelograms of fixed perimeter with angles $\frac{\pi}{5}, \frac{4\pi}{5}$ (more generally, $\frac{\pi}{2m+1}, \frac{2m\pi}{2m+1}$ for fixed $m\in \Z^+$) have the same characteristic polynomial.   Corollary~\ref{cor:intro} is of no help in this case.    In an upcoming paper, we will drop the genericity conditions and address questions of Steklov spectral determination within special classes of convex polygons including triangles, kites and regular polygons.   
In work in progress, we are also investigating the question of whether the characteristic polynomial distinguishes all convex $n$-gons from smoothly bounded simply-connected plane domains.

\subsection{Organization of this work} \label{ss:results}
In \S \ref{s:prelim} we review results of \cite{lpps19} and \cite{klpps21} and provide some simplifications in the context of convex polygons.  We then address eigenvalue bounds in \S \ref{s:ebounds_app}, bounds on the sizes of Steklov isospectral sets of admissible convex $n$-gons in \S 4, and inverse results under weaker genericity conditions in \S 5. We end with a brief comparison between the Steklov and Laplace inverse spectral problems and look towards the future of this field in \S 6.

\section*{Acknowledgements} This work was initiated at the BIRS-CMO workshop 22w5149.  We sincerely thank the organizers as well as all sponsors of the workshop. We thank David Sher, Alexandre Girouard and Iosif Polterovich for inspiring and insightful discussions and correspondence, and Amir Vig for posing the question concerning rational angles discussed in Remark \ref{rem:rat}. C. Villegas-Blas was partially supported by UNAM-PAPIIT-IN 116323  project.

\section{Preliminaries} \label{s:prelim} 
In this section we will recall some of the beautiful results of \cite{klpps21} providing Steklov spectral invariants for simply-connected curvilinear $n$-gons in $\R^2$. The edges of the curvilinear $n$-gons are assumed to be piecewise smooth and the angles at the $n$ vertices are required to lie in the interval $(0,\pi)$.   In the special case in which the edges are geodesic, i.e., the case of convex $n$-gons, we will see that some of their results take on a simpler form.  A polygon with edges that are line segments but that is not necessarily convex will be referred to simply as an $n$-gon.

\subsection{Curvilinear $n$-gons}  We follow the same labeling convention for the edge lengths and interior angles at the vertices as  \cite{klpps21}.

\begin{notaconv}\label{nota:bell and bal} We use $\ell_1,\dots, \ell_n$ to denote edge lengths and $\alpha_1,\dots,\alpha_n$ to denote the interior angles at the vertices.   We will usually abuse notation and use the same notation $\ell_j$, respectively  $\alpha_j$, to denote the $j$th edge, respectively vertex.    In settings where this can result in confusion, we will instead use $e_j$, respectively $v_j$, for the edges and vertices.   We always number the edges and vertices cyclically with vertex $\alpha_j$ occurring between edges $\ell_j$ and $\ell_{j+1}$ (see Figure~\ref{fig:tri_con}); $\ell_{n+1}$ is understood to be $\ell_1$.   The data associated with a curvilinear $n$-gon $\Om$ consists of its vectors of edge lengths and angles
$$\bell=(\ell_1,\dots,\ell_n)\mbox{\,\,\,and\,\,\,}\balpha=(\alpha_1,\dots, \alpha_n).$$
The cyclic labeling is unique only up to $2n$ possible permutations, corresponding to a choice of orientation of $\partial\Om$ and a choice of initial edge.    
\end{notaconv}

 \begin{figure}[h] \centering 
\begin{tikzpicture}
\draw   (-2,0) --  (2,0);
\draw  (-2, 0) -- (0, 3); 
\draw (0, 3) -- (2, 0); 
\node at (-2.2,0) {\small $\alpha_1$}; 
\node at (-1.6, 1.5) {\small $\ell_2$}; 
\node at (0, -0.3) {\small $\ell_1$}; 
\node at (2.3,0) {\small $\alpha_3$}; 
\node at (0, 3.3) {\small $\alpha_2$};
\node at (1.6, 1.5) {\small $\ell_3$}; 
\end{tikzpicture}
\caption{A triangle with angles and edges labeled as in \cite{klpps21}.}
\label{fig:tri_con}
\end{figure}
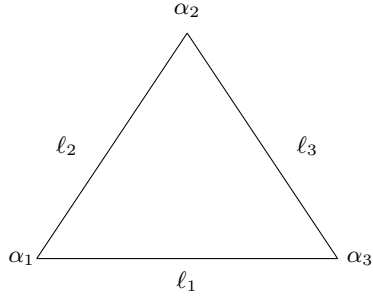

The primary tool we will use to obtain inverse spectral results is the characteristic polynomial $P_\Om$ of a curvilinear $n$-gon introduced in \cite{lpps19}; see also \cite{klpps21}. 

\begin{defn}\label{def: char poly}
The \em characteristic polynomial \em $P_\Om: \R\to \R$
is a trigonometric polynomial given by\em :\em
\beq P_{\Omega} (t) :=\frac{1}{2} \sum_{\bxi\in \{\pm1\}^n}\, a_{\bxi} \cos(|\bxi \cdot \bell| t) - \prod_{j=1} ^n \sin\left( \frac {\pi^2}{2 \alpha_j}\right). \label{eq:charpoly}\eeq
Here $a_{\bxi}$ is defined for $\bxi=(\xi_1,\dots, \xi_n)\in \{\pm1\}^n$ by 
\beq\label{eq:axi} a_{\bxi}=\prod_{\{j:\xi_j\neq \xi_{j+1}\}}\,c(\alpha_j),\eeq
where $a_{\bxi}$ equals $1$ if the product is over the empty set, and 
\beq\label{eq:c(alpha)} c(\alpha_j)=\cos\left(\frac{\pi^2}{2\alpha_j}\right).\eeq
The subscripts in $\bxi$ are cyclically ordered, so $\xi_{n+1}$ is understood to be $\xi_1$.
 Thus in the definition of $a_{\bxi}$, the product is either empty or contains an even number of factors, because there is always an even number of sign changes as one moves cyclically through the entries of $\bxi$ in order to return to the starting value.
\end{defn}

Although it may appear different, this definition is the same as the characteristic polynomial defined in \cite{lpps19} and \cite{klpps21}, which is straightforward to verify. This difference in appearance simplifies several of our proofs. 

Observe that the characteristic polynomial $P_\Om$ depends only on the data $\balpha(\Om)$ and $\bell(\Om)$.   Since $P_\Om$ is an even function, the roots occur in pairs $\pm \nu$.  Let
\beq\label{eq: nu}0\leq \nu_0(P_\Om)\leq \nu_1(P_\Om)\leq \nu_1(P_\Om)\leq\dots\eeq
be all the non-negative roots of $P_\Om$ where the positive roots are repeated according to their multiplicity and zero, if it occurs, is counted with half its multiplicity.  The remarkable main result of \cite{lpps19} is that the roots of the characteristic polynomial determine the asymptotics of the Steklov eigenvalues: 

\begin{thm}\label{thm:lpps}\cite[Theorem 1.4 and Remark 4.21]{lpps19}
Let  $\Om$ be a curvilinear $n$-gon with angles $\alpha_1, \dots ,\alpha_n\in(0,\pi)$.   Then the Steklov eigenvalues $\sigma_j(\Om)$ (see Equation~\eqref{eq:def_stek_ev}) satisfy
\beq\label{eq:asy}\sigma_j(\Om)-\nu_j(P_\Om)=O(j^{-\epsilon})\mbox{\,\,as\,\,}j\to\infty\eeq
for every $\epsilon$ satisfying 
\[0<\epsilon < \min\left(\left\{\frac{\pi}{2\alpha_k}-\frac 1 2: k=1,\dots, n\right\} \cup \left\{\frac 1 4\right\}\right).\] 
\end{thm}
  
In \cite{klpps21}, the authors use the Hadamard-Weierstrass factorization theorem and a result of \cite{ks20} on the zeros of periodic functions to show that $P_\Om$ is uniquely determined by the $o(1)$-asymptotics of its roots, thus yielding:

\begin{thm}\cite[Theorem 1.13 and Remark 1.15]{klpps21}\label{thm:klpps}
Let  $\Om$ and $\Om'$ be curvilinear $n$-gons with all angles in $(0,\pi)$.   Then the following are equivalent:
\begin{itemize}
\item[(a)] $\Om$ and $\Om'$ have the same characteristic polynomial;
\item[(b)] $\sigma_j(\Om)-\sigma_j(\Om')= o(1)\mbox{\,\,as\,\,}j\to\infty$; 
\item[(c)] For $\epsilon$ as in Theorem~\ref{thm:lpps}, we have $\sigma_j(\Om)-\sigma_j(\Om')= O(j^{-\epsilon})\mbox{\,\,as\,\,}j\to\infty$.

\end{itemize}

\end{thm}

The equivalence of (a) and (b) is the content of \cite[Theorem 1.13]{klpps21}. The implication (a)$\implies$(c) follows from  Theorem~\ref{thm:lpps} above, as noted in \cite[Remark 1.15]{klpps21}. Finally, (c)$\implies$(b) is immediate.

\begin{thm}\cite[Theorem 1.16]{klpps21}
    \label{poly is spec invar}
The characteristic polynomial $P_{\Omega}$ of a curvilinear $n$-gon $\Om$ with all angles in $(0,\pi)$ can be constructed algorithmically from the Steklov spectrum of $\Om$.   In particular, the characteristic polynomial is a Steklov spectral invariant of $\Om$.\end{thm}

\begin{remark}\label{rem: smooth} Curvilinear polygons are simply-connected plane domains with piecewise smooth -- but not smooth -- boundary.    In order to use Theorem~\ref{thm:klpps} and Theorem~\ref{poly is spec invar} to compare the spectra of curvilinear $n$-gons to smooth domains, we can extend Definition~\ref{def: char poly} by defining the characteristic polynomial of a smooth plane domain $\Om$ of perimeter $\ell$ to be 
\beq P_\Om(t)=\cos(\ell t) -1. \label{eq:cp_smooth_domain} \eeq 
The sequence of non-negative roots $\nu_j(\Om)$ with multiplicities as in Equation~\eqref{eq: nu} is given in this case by
\[0, \frac{2\pi}{\ell}, \frac{2\pi}{\ell}, \frac{4\pi}{\ell}, \frac{4\pi}{\ell}, \dots\]
which is precisely the Steklov spectrum of a disk of circumference $\ell$.    Moreover, by the well-known Steklov asymptotics for smooth simply-connected plane domains \cite{rozenblum79}, \cite{edward93}, one has 
$\sigma_j(\Om)-\nu_j(\Om) = O(j^{-\infty})$ as $j\to\infty$ for every such domain.  Thus Theorem~\ref{thm:lpps} certainly holds.  Theorem~\ref{thm:klpps} also extends when one includes smooth domains along with curvilinear $n$-gons, since it is based on Theorem~\ref{thm:lpps} together with a proof that the asymptotics of the non-negative roots determine a trigonometric polynomial uniquely.   In particular, this theorem allows one to compare the asymptotics of a given curvilinear $n$-gon with the asymptotics of a smooth domain as in Example~\ref{ex: like smooth} below.    The extension of the definition of the characteristic polynomial to smoothly bounded domains will also be convenient for us in Section~\ref{sec: finiteness}.
\end{remark} 

\begin{exa}\label{ex: like smooth}   Let $\Om$ be a curvilinear $n$-gon satisfying 
\[\balpha(\Om)=\left( \frac{\pi}{2m_1 +1}, \frac{\pi}{2m_2 +1},\dots, \frac{\pi}{2m_n +1}\right).\]
Then one easily computes that 
\[P_\Om(t)=\cos(\ell t) +(-1)^{m+1}\]  
where $\ell$ is the perimeter of $\Om$ and $m:=m_1+\dots+m_n$.    In particular, if $m$ is  
even, then $P_\Om$ has the same characteristic polynomial as a disk and thus the same Steklov spectral asymptotics up to order $O(j^{-\epsilon})$ for all $\epsilon <\frac 1 4$. 
 (See Theorem~\ref{thm:klpps} and Remark~\ref{rem: smooth}.)     
\end{exa}

Observe that $\cos(\ell t)$ necessarily occurs in $P_\Om(t)$ with coefficient one corresponding to $\bxi=\pm(1,1,\dots, 1)$.   This term reflects the well-known fact that the perimeter of a compact planar domain is a Steklov spectral invariant.
Observe that at most $2^{n-1}$ distinct cosine frequencies of the form $\bxi\cdot\bell$ occur in the characteristic polynomial.  As in Example~\ref{ex: like smooth} above, if $c(\alpha_j)=0$, then some of the coefficients $a_{\bxi}$ will vanish.   Information is also lost if $\bxi\cdot\bell=0$ for some $\bxi$, in which case the corresponding cosine function will be absorbed into the constant term of the characteristic polynomial.   If there are repetitions among the various $\bxi\cdot\bell$, then one can have cancellations among their coefficients. 
The article \cite{klpps21} introduces genericity conditions on curvilinear $n$-gons, referred to as \emph{admissibility conditions}, to guarantee that $2^{n-1}$ distinct cosine frequencies appear in the characteristic polynomial with non-zero coefficients. In order to define their genericity conditions, we first introduce the intuitive language of rational, odd, and even angles.

\begin{defn} \label{def:exceptional}
We will say that an angle is \emph{rational} if it is a rational multiple of $\pi$.  Among the rational angles, those of the form $\frac{\pi}{k}$, where $k\in \Z$, will play an especially important role in what follows.     Angles of this form will be called \em odd\em, respectively \em even\em, angles 
if $k$ is an odd, respectively even, positive integer.   \emph{(}These angles are referred to as ``special,'' respectively ``exceptional,'' in \cite{klpps21}.\emph{)}  Observe that an angle $\alpha$ is odd if and only if $c(\alpha)=0$, 
while even angles $\alpha=\frac{\pi}{2m}$ satisfy $c(\alpha)=(-1)^m$.
Following \cite{klpps21}, we will refer to $(-1)^m$ as the \emph{parity} of the even angle $\frac{\pi}{2m}$.  Similarly, we will refer to $(-1)^j$ as the parity of the odd angle $\frac{\pi}{2j+1}$. 
\end{defn} 

\begin{defn}\cite[Definition 1.8]{klpps21} \label{def:generic} A curvilinear $n$-gon with all interior angles in $(0,\pi)$ is said to be \em admissible \em if the following two conditions hold:  (1) the side lengths $\ell_1, \ldots, \ell_n$ are incommensurable over $\{-1, 0, +1\}$ (that is, no non-trivial linear combination of $\ell_1, \ldots, \ell_n$ with coefficients taken from $\{-1, 0, 1\}$ vanishes);  and (2) none of the interior angles $\alpha_1, \ldots, \alpha_n$ are odd    (see Definition~\ref{def:exceptional}). 
\end{defn}

Admissibility can also be viewed as a restriction on the form of the characteristic polynomial.  A set of positive lengths $\ell_1,\dots, \ell_n$ is incommensurable over $\{-1, 0, 1\}$ if and only if all $\bxi\cdot\bell$ where $\bxi\in\{-1,1\}^n$ are distinct and non-zero.  A curvilinear $n$-gon $\Omega$ with all interior angles in $(0, \pi)$ is admissible if and only if its characteristic polynomial $P_\Om$ contains exactly $2^{n-1}$ linearly independent terms of the form $a\cos(ct)$ with $c\neq 0$.   In this case, the constant term in $P_\Om$ will be non-zero if and only if no interior angle of $\Om$ is even. 
 
These observations immediately imply the following proposition, which is a straightforward consequence of \cite{klpps21}. 

\begin{prop}\label{prop:vertices}\cite{klpps21} The characteristic polynomial distinguishes admissible curvilinear $n$-gons from all non-admissible curvilinear polygons that have at most $n$ vertices and have all interior angles in $(0,\pi)$.  Moreover, within the class of all admissible curvilinear polygons, the characteristic polynomial determines the number of vertices.
\end{prop}

We note, however, that an admissible curvilinear $n$-gon may have the same characteristic polynomial as a non-admissible curvilinear polygon with more than $n$ vertices; see Lemma~\ref{om vs omred}.  Before we proceed with stating further inverse spectral results for admissible curvilinear $n$-gons, we recall additional notation from \cite{klpps21}.

\begin{notarem}\label{nota:excep comps}~
Let $\Om$ be a curvilinear $n$-gon with interior angles $\alpha_1,\dots, \alpha_n$.
Write 
\beq \label{eq:pmbC} \pmb{C}(\Om)= (c(\alpha_1),\dots, c(\alpha_n)) \ \ \ \text{and}\ \ \ \pmca(\Om)=(|c(\alpha_1)|,\dots, |c(\alpha_n)|), \eeq 
where $c(\alpha_j)=\cos \left( \frac{\pi^2}{2 \alpha_j} \right)$ as in Definition~\ref{def: char poly}.  Suppose that exactly $k\geq 1$ of the interior angles of $\Omega$ are even.  
The corresponding $k$ vertices split the boundary $\partial\Omega$ into $k$ components $\y_1, \dots, \y_k$ that are referred to as \emph{exceptional components} in \cite{klpps21}.   With a choice of orientation, each exceptional component is described by its vectors of ordered edge lengths and angles:
\beq \label{eq:exceptional_components}
\bell(\y_j)=(\ell_1^j, \dots, \ell_{n_j}^j) \,\,\mbox{and}\,\,\balpha(\y_j)=(\alpha_1^j,\dots \alpha_{n_j-1}^j). \eeq 
Write
\beq \label{eq:C_exceptional_components}
\pmb{C}(\y_j)= \left(c(\alpha^j_1),\dots,  c(\alpha_{n_j-1}^j) \right). \eeq 
   We denote by $-\y_j$ the component $\y_j$ with its orientation reversed.   Following \cite{klpps21}, we refer to $-\y_j$ as the \emph{inverse} of $\y$.
\end{notarem}

We shall repeatedly make use of the following powerful result of \cite{klpps21}.

\begin{thm}\cite[Theorem 1.17]{klpps21} \label{th:ss_ngons}
We use the notation of~\ref{nota:bell and bal}, \ref{def:generic}, and \ref{nota:excep comps}.   
Suppose that $\Omega$ and $\Omega'$ are admissible curvilinear $n$-gons that have the same characteristic polynomial.  Then
\begin{enumerate} 
\item[(a)] $\Omega$ and $\Omega'$ have the same number of even angles.
\item[(b)] If they have no even angles, then the boundary orientations and cyclical labeling of the edges and vertices can be chosen so that 
\[ \bell(\Omega)=\bell(\Omega') \,\,\mbox{and}\,\, \pmb{C}(\Omega)=\pm\pmb{C}(\Omega')\]
for some choice of $\pm$.   
\item[(c)] If there is at least one even angle, then there exists a one-to-one correspondence between the exceptional components of $\Om$ and $\Om'$ such that corresponding exceptional components $\y_j$ and $\y'_j$ satisfy either 
\[\bell(\y_j)=\bell(\y'_j)\,\,\mbox{and}\,\,\pmb{C}(\y_j)=\pm\pmb{C}(\y'_j)\]
or else
\[\bell(\y_j)=\bell(-\y'_j)\,\,\mbox{and}\,\,\pmb{C}(\y_j)=\pm \pmb{C}(-\y'_j)\]
for some choice of $\pm$. 
\end{enumerate}
\end{thm}

Theorem 1.17 in \cite{klpps21} further asserts that for each exceptional component of an admissible curvilinear $n$-gon $\Om$, the characteristic polynomial determines whether the even angles at its two ends have the same or the opposite parity in the sense of Definition~\ref{def:exceptional}.
In part (c) of Theorem \ref{th:ss_ngons}, we emphasize that the one-to-one correspondence does not necessarily respect the order in which the exceptional components appear.   In particular, adjacent exceptional components in $\Om$ need not correspond to adjacent ones in $\Om'$.

\begin{remark}\label{rem: rem on thm}~
 If $\Om$ has precisely one even angle, then there is only one exceptional component $\y$.   Reorienting $\y$ is equivalent to simply reorienting $\pa\Om$ and thus is a trivial change, i.e., it does not affect the isometry class.  Part (c) of the theorem implies in this case that the boundary orientation of $\Om'$ can be chosen so that $\bell(\Om)=\bell(\Om')$ and $\pmb{C}(\y)=\pm\pmb{C}(\y')$.   Thus, up to the choice of boundary orientation and cyclic labeling, the characteristic polynomial determines $\pm\pmb{C}(\Om)$ up to the sign of the entry $\pm 1$ (the entry corresponding to the even angle) and up to a global sign change of all the remaining entries.   In particular, it determines $\pmca(\Om)$ uniquely up to trivial changes.
\end{remark}

\begin{remark}\label{rem:rat}
Amir Vig raised the following question to us:  Does the Steklov spectrum of an $n$-gon detect whether the angles are rational multiples of $\pi$? Theorem~\ref{th:ss_ngons} yields a positive answer in the case of admissible curvilinear $n$-gons: if $\Om$ and $\Om'$ are Steklov isospectral admissible curvilinear $n$-gons and if all of the interior angles of $\Om$ are rational multiplies of $\pi$, then the same is true for all the angles of $\Om'$.   Indeed, Theorem~\ref{th:ss_ngons} tells us that, up to reordering, $|c(\alpha_j)|=|c(\alpha'_j)|$ for $j=1,\dots, n$.   This implies that $\frac{\pi^2}{2\alpha'_j}=k\pi \pm \frac{\pi^2}{2\alpha_j}
$ for some $k\in \Z$.  Writing $\alpha_j=q_j\pi$, we then have $\alpha'_j=\frac{q_j\pi}{2kq_j\pm 1}.$
\end{remark}

To conclude the background on curvilinear $n$-gons, we summarize the properties of $|c(\alpha)|$ that will be used extensively. 
\begin{lemma}\label{lem: |c|}~
Define $|c|:(0,\pi)\to [0,1]$ by $|c|(\alpha):=|c(\alpha)|$ where $c(\alpha)=\cos\left(\frac{\pi^2}{2\alpha}\right)$ as in Definition~\ref{def: char poly}.  Then:
\begin{enumerate}
\item[(a)]  $|c|^{-1}(\{0\})$ consists of all odd angles $\frac{\pi}{2k+1}$, $k\in \Z^+$.
\item[(b)] $|c|^{-1}(\{1\})$ consists of all even angles $\frac{\pi}{2k}$, $k\in \Z^+$.
\item[(c)] $|c|$ maps each interval $(\frac{\pi}{m+1},\frac \pi m)$, $m\in \Z^+$, bijectively onto $(0,1)$.     In particular, the restriction of $|c|$ to the set of all obtuse angles is injective.
\item[(d)] For $s\in [0,1]$, the inverse image $|c|^{-1}(\{s\})$ is discrete and accumulates only at $0$.
\end{enumerate}
\end{lemma}

\subsection{Convex polygons}
We now specialize to the case of convex polygons $\Om$; i.e., in addition to assuming that all angles lie in $(0, \pi)$, we assume all edges of $\Om$ are straight line segments.  The only convex $n$-gon that has three odd angles is the equilateral triangle; all others have at most two odd angles since the angles must sum to $(n-2)\pi$.  For the same reason, with the exception of rectangles, a convex $n$-gon can have  at most three even angles. In particular, an admissible convex $n$-gon $\Om$ can have at most three exceptional components.  Consequently, any two exceptional components are adjacent, so we can view any ordering $\y_1,\dots, \y_k$ of the exceptional components as a cyclic ordering.

\begin{notaconv} A choice of orientation of $\pa\Om$ induces compatible orientations of each exceptional boundary component.   Moreover, the orientation yields a cyclic ordering $\y_1,\dots,\y_k$ of the boundary components, unique up to the choice $\y_1$.   In what follows, we will always assume that orientations and cyclic ordering of the exceptional boundary components are simultaneously compatible with some orientation of $\pa\Om$.   
\end{notaconv}

Thus part (c) of Theorem~\ref{th:ss_ngons} takes on the following simpler form:

\begin{cor}\label{cor:excep} Suppose that $\Omega$ and $\Omega'$ are admissible convex $n$-gons that have the same characteristic polynomial and that have $k>0$ even angles.   Then there exist orientations of $\pa\Om$ and $\pa\Om'$ and cyclic orderings $\y_1,\dots,\y_k$ and $\y'_1,\dots\y'_k$ of the exceptional components compatible with the orientations of $\pa \Omega$ and $\pa\Omega'$, respectively, so that
for each $j\in \{1,\dots, k\}$, we have either
\[\bell(\y_j)=\bell(\y'_j)\,\,\mbox{and}\,\,\pmb{C}(\y_j)=\pm\pmb{C}(\y'_j)\]
or else
\[\bell(\y_j)=\bell(-\y'_j)\,\,\mbox{and}\,\,\pmb{C}(\y_j)=\pm\pmb{C}(-\y'_j)\]
for some choice of $\pm$. 

\end{cor}

The corollary is immediate from Theorem~\ref{th:ss_ngons} since every ordering of the exceptional components is cyclic and compatible with some orientation of the boundary.

\section{Eigenvalue bounds and applications to Steklov isospectrality} \label{s:ebounds_app}

In this section, we demonstrate a collection of estimates for the Steklov eigenvalues.   In Subsection~\ref{subsec:passage}, we develop the tools needed for the rest of the section.  In particular, we extend work of Girouard and Polterovich \cite{zhir_pol} addressing Steklov eigenvalue bounds for Riemannian surfaces containing long thin passages.  In Subsection~\ref{subsec:triangle bounds}, we obtain bounds for the perimeter-normalized Steklov eigenvalues of arbitrary triangles in terms of the smallest vertex angle.  Turning to $n$-gons with $n\geq 4$ in Subsection~\ref{subsec:n-gons}, we first obtain Steklov eigenvalue bounds for long, thin $n$-gons, \emph{convex or not}.  Then as a consequence, we obtain bounds for the perimeter-normalized Steklov eigenvalues of \emph{convex} $n$-gons $\Omega$ in terms of the smallest vertex angle $\alpha(\Omega)$, provided $\alpha(\Omega)$ is sufficiently small.

\subsection{Riemannian surfaces containing long thin passages}\label{subsec:passage}

Steklov eigenvalues satisfy a certain variational principle, also known as a min-max principle, which allows one to obtain eigenvalue estimates by choosing specific test functions. This variational principle can be shown in a very general context (see \cite{bir_sol_70}), but the following formulation will suffice for our purposes.

\begin{prop} \label{prop:rayleigh} Let $\Omega$ be a compact Riemannian manifold with boundary. For $u\in H^1(\Omega)$, the Rayleigh quotient for the Steklov problem is defined by
\[ R(u)=\frac{\int_\Omega\,|\nabla u|^2 dA}{\int_{\partial\Omega}\,u^2 ds}. \] 
Here, $dA$ denotes the Riemannian volume form on $\Omega$, and $ds$ the induced Riemannian measure on the boundary. Let $\mathcal E_k(\Omega)$ denote the set of all $k$-dimensional subspaces of $H^1(\Omega)$ that consist of functions whose restrictions to $\partial\Omega$ are orthogonal to constants relative to the $L^2(\partial \Omega)$ inner product.  Then, the Steklov eigenvalues satisfy 
\beq \sigma_k(\Omega)=\min_{E\in \mathcal E_k(\Omega)}\,\max_{0\neq u\in E}\, R(u). \label{eq:rayleigh_vp} \eeq  
\end{prop}   

Girouard and Polterovich \cite[\S 2]{zhir_pol} gave Steklov eigenvalue bounds for compact Riemannian manifolds of arbitrary dimension that contain a long thin cylindrical passage.    We state their result only in dimension two and then, still in the 2-dimensional case, we give two extensions, the first in Proposition~\ref{prop:trunc sector} and the second in Proposition~\ref{prop:quadrilat}.

\begin{lemma}\label{lemma:Gir Polt}\cite[\S 2]{zhir_pol} Let $\Omega$ be a compact Riemannian surface with Lipschitz boundary that contains a Euclidean rectangle of length $\ell$ and width $w$.   Assume that the two sides of length $\ell$ lie in $\partial \Omega$.   Then the $k^{th}$ Steklov eigenvalue of $\Omega$ satisfies  
\beq \sigma_k(\Omega) \leq  \frac{2\pi^2 k^2 w}{\ell^2}. \label{eq:lemmaGirPolt} \eeq 
\end{lemma}
We note that there is no additional hypothesis on the sides of length $w$; they may or may not lie in $\partial \Omega$.

\begin{proof}
One uses the variational characterization of eigenvalues in Proposition~\ref{prop:rayleigh}.   
Without loss of generality, we may assume that the rectangle contained in $\Omega$ is located at $[0, \ell] \times [0, w]$ in the $xy$-plane.
We define test functions on the rectangle via 
\beq u_j(x,y) =  \sin\left(\frac{2\pi j x}{\ell} \right), \quad 0 \leq x \leq \ell, \quad 0 \leq y \leq w, \label{eq:ujR} \eeq
and extend $u_j\equiv 0$ outside the rectangle.   Then $E_k:=\textrm{span}\{u_1,\dots, u_k\}\in \cE_k(\Omega)$ with $\cE_k$ as in Proposition \ref{prop:rayleigh}.  We have
\beq \int_\Omega |\nabla u_j|^2 dA &=& \frac{4\pi^2 j^2}{\ell^2} \frac{w \ell}{2}.
\nonumber \eeq 
Moreover, $\nabla u_j$ is orthogonal to $\nabla u_m$ in $L^2(\Omega)$ for $j \neq m$.  We compute 
\[ \int_{\pa \Omega} u_j \, u_m \, ds = \begin{cases} 0, & j \neq m, \\ \ell, & j = m. \end{cases} \]
Therefore, for every real linear combination $u=a_1 u_1 + \ldots + a_k u_k \in E_k$, we have
\[
R(u) \leq \frac{2\pi^2 w}{\ell^2} \frac{\sum_{j=1} ^k j^2 a_j^2}{\sum_{j=1} ^k a_j^2} \leq \frac{2\pi^2k^2 w}{\ell^2},
\]
which implies that 
\[ \sigma_k(\Omega) \leq \frac{2\pi^2k^2 w}{\ell^2}. \] 
\end{proof}

  The actual eigenvalue bound in the lemma above is not explicitly stated in \cite{zhir_pol} but the test functions are given there.    The lemma does not require that $\ell \gg w$ but the bounds are much stronger  in that case.

\begin{defn} \label{def:sector} Recall that a \emph{polar rectangle} is a sector either of a Euclidean disk or of a Euclidean annulus.  If $r_1$ and $r_2$ are the inner and outer radii (so $r_1=0$ in the case of a disk sector), we will refer to $L:=r_2-r_1$ as the \emph{radial side length}.   
\end{defn}

In the next proposition, we show how to use these polar rectangles to obtain estimates in the spirit of Lemma \ref{lemma:Gir Polt}.  
\begin{prop}\label{prop:trunc sector}~
Let $\Omega$ be a compact Riemannian surface with Lipschitz boundary that contains a closed subdomain $S$ isometric to a polar rectangle of radial side length $L$ and opening angle $\alpha$.  Let $0\leq r_1 < r_2$ be the inner and outer radii (thus $L=r_2-r_1$) and let $s_1$ and $s_2$ be the arclengths of the inner and outer circular boundary arcs. Suppose that the two radial boundary edges of $S$ lie in $ S\cap \partial\Omega$.  (We make no assumption on whether the inner and outer circular edges lie in $\partial\Omega$.)  Then for all $k=1,2,\dots$, the Steklov eigenvalues of $\Omega$ satisfy 
\begin{equation}\label{eq:upper boundv2}\sigma_k(\Omega) \leq \alpha \frac{k^2 \pi^2}{L} \left[ 1 + \frac{2 r_1}{L}\right]=k^2\pi^2\left(\frac{s_1+s_2}{L^2}\right).\end{equation}
In particular, if $r_1=0$ (i.e., $S$ is isometric to a sector of a disk), then 
\begin{equation}\label{eq:upper bound}\sigma_k(\Omega)\leq \pi^2\frac{k^2\alpha}{L}.\end{equation}
\end{prop}

\begin{proof}
   The second statement follows from the first since $s_1=0$ and $s_2=\alpha L$ when $r_1=0$.  To prove the first, we again apply the variational principle \eqref{eq:rayleigh_vp}.  We assume without loss of generality that the polar rectangle is described by $r_1 < r < r_2$ and $0<\theta < \alpha$.  Using these polar coordinates $(r, \theta)$ on $S$, define functions $u_j$ on $S$ by 
\beq  u_j(r,\theta) =  \sin\left(\frac{2\pi j\, (r-r_1)}{L}\right) \label{eq:uj}, \eeq
and extend $u_j$ to $\Omega$ by setting $u_j\equiv 0$ on $\Omega\setminus S$.   
We have $E_k=\textrm{span}\{u_1,\dots, u_k\}$ $\subset \cE_k$ with $\cE_k$ as in Proposition \ref{prop:rayleigh}; the functions $u_j$ satisfy 
\begin{equation}\label{eq:gradint1} \int_\Omega |\nabla u_j|^2 dA = \alpha \frac{4\pi^2 j^2}{L^2} \int_{r_1}^{r_2}  |\cos(2\pi j (r-r_1)/L)|^2 r dr = \alpha \pi^2 j^2 \left[ 1+ \frac{2r_1}{L} \right].
\nonumber\end{equation}
We compute that $\nabla u_j$ is orthogonal to $\nabla u_m$ in $L^2(\Omega)$ for $j \neq m$ and
\[ \int_{\pa \Omega} u_j \, u_m \, ds = 2 \int_{r_1} ^{r_2} \sin(2\pi j (r-r_1)/L) \sin(2 \pi m (r-r_1)/L) dr = \begin{cases} 0, & j \neq m, \\ L, & j = m. \end{cases} \]
Therefore, for every real linear combination $u=a_1 u_1 + \ldots + a_k u_k \in E_k$, we have
\[
R(u) \leq \alpha \pi^2 \left[ 1+ \frac{2r_1}{L} \right] \frac{\sum_{j=1} ^k j^2 a_j^2}{L \sum_{j=1} ^k a_j^2} \leq \alpha \frac{k^2 \pi^2}{L} \left[ 1 + \frac{2 r_1}{L} \right], \]
giving the desired upper bound.  

The final equality follows from the facts that $L=r_2-r_1$ and that $s_i=r_i\alpha$ for $i=1,2$.
\end{proof}

We build upon Lemma \ref{lemma:Gir Polt} and Proposition \ref{prop:trunc sector} to obtain eigenvalue estimates for Riemannian surfaces that contain either a long and narrow quadrilateral or a long and narrow triangle.  

\begin{prop}\label{prop:quadrilat} Let $\Omega$ be a compact Riemannian surface with Lipschitz boundary.
\begin{enumerate}[(a)] 
\item Suppose that $\Omega$ contains a long, thin Euclidean quadrilateral $Q$ with vertices in cyclic order given by $p_1, q_1, q_2, p_2$.    More precisely, writing $$w:= \max\{|p_1p_2|, |q_1q_2|\}$$ and $$\ell:= \min\{|p_1q_1|,|p_2q_2|\},$$ suppose that  $$\ell>3w$$
as in Figure~\ref{fig:quadrilat}.  Assume that the two long sides $p_1q_1$ and $p_2q_2$ lie in $\partial \Omega$.   Then the $k^{th}$ Steklov eigenvalue of $\Omega$ satisfies
$$\sigma_k(\Omega)\leq 2k^2 \pi^3\frac{w}{(\ell-3w)^2}.$$

\item Suppose that $\Omega$ contains a Euclidean triangle $T$ with vertices $p,q_1,q_2$, such that the sides $pq_1$ and $pq_2$ lie in $\partial\Omega$ and that
$$w:=|q_1q_2|< \frac{\ell}{2}<\ell = \min\{|pq_1|,|pq_2|\}.$$
Then the $k^{th}$ Steklov eigenvalue of $\Omega$ satisfies 
$$\sigma_k(\Omega)\leq k^2 \pi^3\frac{w}{(\ell-2w)^2}.$$
\end{enumerate}
\end{prop}

\begin{figure}[h] \centering 
 \begin{tikzpicture}
\draw   (3.8,0) --  (16.4,0);
\draw (4,0.25) -- (16,1); 
\draw (3.8,0) -- (4,0.25) ;
\draw   (16.4,0) --  (16,1);
\draw [dashed] (0,0) -- (3.8,0); 
\draw [dashed] (0,0) -- (4,0.25); 
\node at (-0.2,0) {\small $v$}; 
\node at (3.8,-0.3) {\small $p_1$}; 
\node at (16.4, -0.3) {\small $q_1$}; 
\node at (4.0, 0.5) {\small $p_2$};
\node at (16,1.3) {\small $q_2$};
\end{tikzpicture} 
\caption{A long thin quadrilateral $Q$. The extensions of the two long sides of $Q$ intersect at $v$.} \label{fig:quadrilat} 
\end{figure}
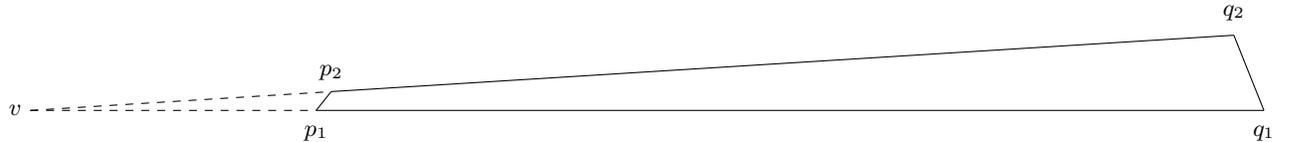 

\begin{proof}
(a) Assume first that the two long sides are parallel, that is $p_1 q_1$ is parallel to $p_2 q_2$.  Then the distance between these sides is bounded above by $w$.  We slice off a small region of $Q$ near each of the two short sides in order to obtain a rectangle of length at least $\ell-|p_1p_2|-|q_1q_2|\geq \ell-2w$ and width $\leq w$.  We then apply Lemma \ref{lemma:Gir Polt} to complete the proof in this case.   

Thus we assume that $p_1q_1$ is not parallel to $p_2q_2$.
We will construct a polar rectangle in $\Omega$ satisfying the hypotheses of Proposition~\ref{prop:trunc sector}.  Take an isometric copy of $Q$ in $\R^2$ and let $v$ be the point of intersection of the lines through $p_1q_1$ and $p_2q_2$ as in Figure~\ref{fig:quadrilat}.  We may assume for convenience that $|p_1p_2| < |q_1q_2|$.  Thus $p_i$ is the closest point to $v$ and $q_i$ the furthest point from $v$ on side $p_iq_i$ for $i=1,2$.   Let 
\[ r_1:= \max (|v p_1|, |v p_2|)  \]
and 
\[ r_2':= \min( |v q_1|, |v q_2|).\]
(We are using the notation $r_2'$ here as we will shrink it below to obtain the outer radius $r_2$ of the desired polar rectangle.)
Let $S_v(t)$ denote the circle with center $v$ and radius $t$.   
Then $S_v(r_1)$, respectively $S_v(r_2')$, intersects side $p_iq_i$ at a point $p_i'$ within distance  $|p_1p_2|$ of $p_i$, respectively a point $q_i'$ within distance $|q_1q_2|$ of $q_i$, for $i=1,2$.  (Note that $p_i'=p_i$ and $q_j'=q_j$ for at least one value of $i$ and one value of $j$ in $\{1,2\}.$)   Thus 
\begin{equation}\label{r2'-r1} r_2'-r_1\geq \ell-|p_1p_2|-|q_1q_2|> \ell-2w.\end{equation} 
  
We next shrink $r_2'$ since the polar rectangle centered at $v$ with inner radius $r_1$ and outer radius $r_2'$ may extend a little outside of $Q$ near edge $q_1q_2$.   Denote by $r_2$ the distance from $v$ to $q_1'q_2'$.  Then the polar rectangle $S$ bounded by $S_v(r_1)$, $S_v(r_2)$, $p_1q_1$ and $p_2q_2$ lies entirely inside $Q$. 

Observe that for any $t\in [r_1,r_2']$, the chord of the circle $S_v(t)$ joining points on $p_1q_1$ and $p_2q_2$ has length at most $2w$. 
In particular, the midpoint $q_m$ of the chord $q_1'q_2'$ satisfies $|q_1'q_m|\leq w$.    Thus $r_2 \geq r_2'-w$ and by Inequality~\eqref{r2'-r1}, we have 
\begin{equation}\label{r_2-r_1}L:=r_2-r_1 > \ell-3w.\end{equation}
Next, since the length $s$ of the arc of a circle subtended by a chord of length $c$ satisfies $s\leq \frac{\pi}{2} c$, the inner and outer arclengths $s_1$ and $s_2$ of $S$ satisfy 
\begin{equation}\label{s1 and s2}s_j\leq \frac{\pi}{2} (2w) =\pi w\end{equation}
for $i=1,2$.   Applying Proposition~\ref{prop:trunc sector}, we thus have 
$$\sigma_k(\Omega)\leq 2k^2 \pi^3\frac{w}{(\ell-3w)^2},$$
completing the proof.

(b) The proof follows the same steps with some minor modifications. We now set $p_1=p_2=p$, so $v=p$, and $r_1=0$.  Inequality~\eqref{r2'-r1} becomes $r_2'\geq \ell-|q_1q_2|= \ell-w$. Since $r_2\geq r_2'-w$ as before, we have $L:=r_2\geq \ell-2w$.   Finally $s_1=0$ and, as before, $s_2\leq \pi w$.   We can now apply the bound in \eqref{eq:upper bound} in Proposition~\ref{prop:trunc sector} to obtain the stated  eigenvalue bounds.
\end{proof}

\subsection{Steklov eigenvalue bounds for triangles}\label{subsec:triangle bounds}
We apply the results of the previous subsection to give bounds for the perimeter-normalized Steklov eigenvalues of triangles.  The bounds depend only on the smallest angle of the triangle.  Note the contrast with the second item in Proposition~\ref{prop:quadrilat}, which does not require that the domain itself be a triangle but imposes assumptions on the lengths of the sides of the triangular subdomain. We will prove the eigenvalue bound for triangles by using Proposition~\ref{prop:trunc sector}, independently of Proposition~\ref{prop:quadrilat}. 

\begin{prop} \label{prop:angle_estimate}
Let $T$ be a triangle and denote by $\alpha(T)$ its smallest interior angle.   Then
\[ \sigma_k(T)L(\partial T) \leq \frac{8\sqrt{3}}{3} \pi^2 k^2 \alpha(T),\]
with $L(\partial T)$ the perimeter of $T$. \end{prop}

\begin{proof} We will assume that $L(\partial T)=1$ and obtain the general case from the fact that if we scale a domain by a constant factor $c$, the Steklov eigenvalues scale by $c^{-1}$.       
We will show that for any triangle $T$ of perimeter one and with smallest interior angle $\alpha(T)$, there exists a universal constant $L>0$ such that the intersection of $T$ with a disk of radius $L$ centered at the vertex $\alpha(T)$ is a sector of angle $\alpha(T)$ and radius $L$.  Moreover, we will show that $L \geq \frac{\sqrt{3}}{8}$. The proposition will then follow from the bound in \eqref{eq:upper bound} in Proposition~\ref{prop:trunc sector}.

 If there is more than one vertex with angle $\alpha(T)$, choose one.  
Let $A,B,C$ be the sides of $T$ with $A\leq B\leq C$.  Then the vertex opposite side $A$ has angle $\alpha(T)$.  Since the perimeter of $T$ is one, the triangle inequality implies that $A+B> \frac{1}{2}$.  So, $\frac{1}{4} < B\leq C $.  Taking $L=\frac{1}{4}$ almost works but not quite; the curved part of the boundary of the sector could potentially leave $T$.   However, $L=\frac{1}{4}\cos(\alpha/2)$ will work; see Figure \ref{fig:var_est}.  
Observe that since $\alpha(T)$ is the smallest angle, and the angles sum to $\pi$, we have $\alpha(T)\leq \frac{\pi}{3}$, so $\cos(\alpha/2) \geq\frac{\sqrt{3}}{2}$.  Thus $L=\frac{\sqrt{3}}{8}$ works for all triangles and we can apply the bound \eqref{eq:upper bound} of Proposition~\ref{prop:trunc sector} to complete the proof.
\end{proof}

  \begin{figure} 
 \begin{tikzpicture}
\draw   (0,0) circle[radius=13.4mm];
\fill   (0,0) circle[radius=1pt];  
\draw   (0,0) --  (1.5,0);
\draw (0,0) -- (1, 1); 
\draw (1, 1) -- (1.5, 0); 
\node at (0.5,0.2) {\small $\alpha$}; 
\node at (1.4,0.7) {\small A}; 
\node at (0.35,0.68) {\small B}; 
\node at (0.8,-0.25) {\small C}; 
\end{tikzpicture} 
\caption{This triangle has side lengths $A \leq B \leq C$, and with a suitable choice of units, $A+B+C=1$.  By the triangle inequality, we therefore have $1/4 < B, C$. 
Consequently, a circular sector of radius $L \leq 1/4 \cos (\alpha/2)$ with opening angle $\alpha$ fits inside the triangle and intersects the triangle only along the straight edges of the sector.} \label{fig:var_est}
\end{figure}
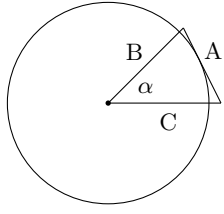  

We obtain a better bound for the even Steklov eigenvalues in the case of isosceles triangles in which the two equal angles are smaller than the remaining angle.  

\begin{cor} \label{cor:rayleigh_isotri} 
Let $T$ be an isosceles triangle such that the two equal angles of measure $\alpha$ are less than or equal to the remaining angle.  Then the even perimeter-normalized Steklov eigenvalues satisfy
\[\sigma_{2k}(T)L(\pa T) \leq {\pi^2 k^2}\frac{2(1+\cos(\alpha))}{\cos(\alpha)}\alpha \leq 6 \pi^2 k^2 \alpha.\] 
\end{cor} 

\begin{proof} Since the two equal angles are smaller than the remaining angle, $\alpha(T)\leq \frac{\pi}{3}$.   Letting $A$ denote the length of each of the two equal sides, we have $2A+2A\cos(\alpha)=$ perimeter of $T$.  Without loss of generality, we assume the perimeter of $T$ equals 1.   So 
\[ A=\frac{1}{2(1+\cos(\alpha))}.\]   
The altitude through the remaining angle (the largest angle) bisects the base, with each half having length 
\[ L:=A\cos(\alpha)= \frac{\cos(\alpha)}{2(1+\cos(\alpha))}.\]  
The 2 sectors of angle $\alpha$ and length $L$ emanating from the 2 vertices of angle $\alpha$ intersect only at the midpoint of the longest side of the triangle as shown in Figure \ref{fig:ray_iso}. 

\begin{figure}[h] \centering 
 \begin{tikzpicture}
\draw   (0,0) --  (6,0);
\draw (0,0) -- (3, 1); 
\draw (3, 1) -- (6,0);
\draw [dashed] (3,1) -- (3,0); 
\node at (-0.2,0) {\small $\alpha$}; 
\node at (6.3, 0) {\small $\alpha$}; 
\node at (1.5, -0.3) {\small L}; 
\node at (4.5, -0.3) {\small L};
\draw[blue] (3,0) arc[start angle=0, end angle=18.435, radius=3];
\draw[red] (3,0) arc[start angle=180, end angle=161.565, radius=3];
\end{tikzpicture} 
\caption{This isosceles triangle has two smaller equal angles of measure $\alpha$. The two circular sectors of radius $L$ centered at the two vertices of angle $\alpha$ intersect only at the midpoint of the side of length $2L$.} \label{fig:ray_iso} 
\end{figure}
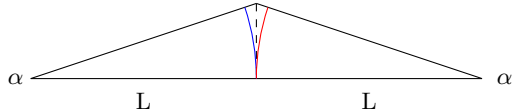 
 
Thus to estimate $\sigma_{2k}$, we can use $E_{2k}=\{u_1,\dots, u_k,v_1,\dots, v_k\}$ where the $u_j$'s, respectively $v_j$'s, are defined according to \eqref{eq:uj} with support on the first, respectively second, sector.  Here we are setting $r_1=0$ in \eqref{eq:uj}.   We then obtain
\[ \sigma_{2k}(T)\leq \frac{\pi^2 k^2}{L}\alpha={\pi^2 k^2}\frac{2(1+\cos(\alpha))}{\cos(\alpha)}\alpha. \] 
To complete the proof we note that $\frac{1 + \cos(\alpha)}{\cos(\alpha)}$ is an increasing function of $\alpha$ on $(0, \frac{\pi}{3})$ that has value $3$ at $\alpha = \frac{\pi}{3}$.
\end{proof} 

\begin{remark} 
 Setting $L(\partial T) = 1$, the estimate from Proposition \ref{prop:angle_estimate} gives 
\[ \sigma_{2k}(T) \leq \frac{32 \sqrt 3}{3} \pi^2 k^2 \alpha \approx 18.5\pi^2 k^2 \alpha ,  \] 
in comparison to $6 \pi^2 k^2 \alpha$ from the corollary. 
\end{remark} 

\subsection{Steklov eigenvalue bounds for $n$-gons}\label{subsec:n-gons}
We first give eigenvalue bounds for long thin $n$-gons that are \emph{not} necessarily convex as shown in Figure \ref{fig:nonconvex}.

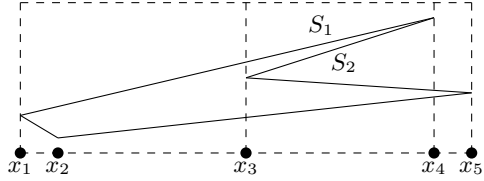
\begin{figure}[h] \centering 
\begin{tikzpicture}
\draw[dashed]   (-3,-1) --  (3,-1);
\draw[dashed]   (-3,-1) --  (-3,1);
\draw[dashed]   (-3,1) --  (3,1);
\draw[dashed]   (3,1) --  (3,-1);
\draw  (-3, -0.5) -- (2.5, 0.8); 
\draw  (-3, -0.5) -- (-2.5, -0.8); 
\draw  (-2.5, -0.8) -- (3, -0.2); 
\draw (3, -0.2) -- (0, 0); 
\draw (0, 0) -- (2.5, 0.8); 
\node at (-3, -1)[circle,fill,inner sep=1.5pt]{}; 
\node at (-3, -1.2) {\small $x_1$};
\node at (-2.5, -1)[circle,fill,inner sep=1.5pt]{}; 
\node at (-2.5, -1.2) {\small $x_2$};
\node at (0, -1)[circle,fill,inner sep=1.5pt]{}; 
\node at (0, -1.2) {\small $x_3$};
\node at (2.5, -1) [circle,fill,inner sep=1.5pt]{}; 
\node at (2.5, -1.2) {\small $x_4$};
\node at (3, -1) [circle,fill,inner sep=1.5pt]{}; 
\node at (3, -1.2) {\small $x_5$};
\draw[dashed]   (0,-1) --  (0,1);
\draw[dashed]   (2.5,-1) --  (2.5,1);
\node at (1, 0.7) {\small $S_1$}; 
\node at (1.3, 0.2) {\small $S_2$}; 

\end{tikzpicture}
\caption{The x-coordinates of this polygon are labelled from left to right. 
We create the dashed rectangle with vertices $x_3$ and $x_4$.  Then there are an even number of disjoint open segments in the boundary of the polygon whose closures have endpoints with $x$-coordinates equal to $x_3$ and $x_4$, respectively.  The topmost are denoted $S_1$ and $S_2$.}
\label{fig:nonconvex}
\end{figure} 

\begin{prop}\label{prop:rect bound} Let $\Omega$ be an $n$-gon contained in a rectangle $[0,\ell^*]\times [-\frac{w^*}{2}, \frac {w^*}{2}]$ with $w^* < \frac{\ell^*}{3(n-1)}$.   Assume that at least one vertex of $\Omega$ lies on each of the sides $x=0$ and $x=\ell^*$.   Then 
$$\sigma_k(\Omega)\leq \frac{2k^2(n-1)^2\pi^3 w^*}{(\ell^*-3(n-1)w^*)^2}.$$

\end{prop}
There are no assumptions on the perimeter of $\Omega$, although the hypotheses imply that $|\partial \Omega|>2\ell^*$.  

\begin{proof}
Let $\{x_1,\dots, x_m\}$ be the set of all $x$-coordinates of vertices of $\Omega$, labelled so that $0=x_1<x_2\dots <x_m=\ell$.   There may be more than one vertex with a given $x$-coordinate, so $m$ can be less than $n$.  We emphasize that the labelling of the $x_i$'s does not coincide with the usual cyclical labelling of vertices.     Since $m\leq n$, at least one index $i\in \{2,\dots, m\}$ satisfies $x_i-x_{i-1}\geq \frac{\ell^*}{n-1}$.  Fix such an $i$.

The subrectangle $$R_i:=(x_{i-1},x_i)\times (-\frac{w^*}{2},\frac{w^*}{2})$$ intersects $\partial \Omega$ in an even number of disjoint open segments $S_j$, each of whose closures $\overline{S}_j$ has endpoints on the two edges $\{x_{i-1}\}\times [-\frac {w^*} 2, \frac {w^*} 2]$  and $\{x_{i}\}\times [-\frac {w^*} 2, \frac {w^*} 2]$.  This is depicted in Figure \ref{fig:nonconvex}.
In general, the evenness follows from there being no vertices with $x$-coordinates contained in $(x_{i-1}, x_i)$.  Since the polygon is not collapsed, for each part of the boundary contained in this subrectangle there is an opposing segment, hence the segments come in pairs. Moving vertically down from the top of the subrectangle, one enters $\Omega$ upon crossing the highest segment (call it $S_1)$, exits $\Omega$ upon crossing the next one $S_2$, and so forth.    If $\Omega$ is convex, there are exactly two such segments; otherwise there can be more than two but we will focus just on the first two in what follows.  The region $Q$ in $R_i$ between $S_1$ and $S_2$ is either a quadrilateral or a triangle.   In either case, we can apply Proposition~\ref{prop:quadrilat} with $\ell=\frac{\ell^*}{n-1}$ and $w=w^*$ to obtain 
\[\sigma_k(\Omega)\leq \frac{2k^2\pi^3w}{(\ell-3w)^2}=\frac{2k^2(n-1)^2\pi^3 w^*}{(\ell^*-3(n-1)w^*)^2}.\]
\end{proof}

With the preceding result, we can partially generalize the eigenvalue bound for triangles to all convex polygons.  

\begin{thm} \label{th:awesome}
For $n=3,4,5,\dots$ and for 
$\delta_n>0$ sufficiently small, there exists a constant $C_n>0$ depending only on $n$ and $\delta_n$ such that if $\Omega$ is any convex $n$-gon  with smallest angle $\alpha(\Omega)<\delta_n$, then the Steklov eigenvalues of $\Omega$ satisfy 
$$\sigma_k(\Omega) L(\partial \Omega) \leq C_n k^2\alpha(\Omega), \quad \text{for all } k \geq 0.$$  
In particular, it is sufficient to assume that 
\begin{equation}\label{eq:delta n}\delta_n < \frac{0.98}{3n-2},\end{equation}
and we can take 
\begin{equation}\label{eq:C n}C_n = \frac{(n-1)^2\pi^3 }{(0.98)(\frac{1}{2}-(3n-2)\frac{\delta_n}{2(0.98)})^2}.\end{equation}
\end{thm}

\begin{proof} 
We place $\Omega$ so that the vertex of its smallest interior angle, say of measure $\alpha$, is at the origin, and the horizontal axis ($x$-axis) bisects this angle.   Let $\ell$ be the maximum distance of the vertices of $\Om$ from the $y$-axis, and assume without loss of generality that the perimeter of $\Omega$ is one.   Then $\ell < \frac 1 2$.  By convexity, $\Omega$ is contained in the isosceles triangle with vertices $(0,0)$ and $(\ell, \pm \ell\tan(\alpha/2))$ as in Figure~\ref{fig:placeP}.   
Moreover, since $\ell<\frac{1}{2}$, the polygon $\Omega$ lies in a rectangle $R$ of length $\ell$ and width $w:=\tan(\alpha/2)$ as in Figure~\ref{fig:placeP}; the perimeter of $R$ is greater than the perimeter of $\Omega$, i.e., greater than one.    Thus $2\ell+2w>1$ and 
  \begin{equation}\label{eq:ell bd}\ell>0.5 -w.\end{equation}   
 
 To obtain an upper bound for $w$, assume that $\alpha<\frac{1}{7}$.  (This will be the case if $\alpha<\delta_n$, where $\delta_n$ is given as in the statement of the theorem.)  The Maclaurin series for the cosine then implies that $\cos(\alpha/2)>\frac{97}{98}>0.98$.  Thus
\begin{equation}\label{eq:tan value}w=\tan(\alpha/2)<\frac{\sin(\alpha/2)}{0.98}<\frac{\alpha}{2(0.98)}.\end{equation}

  To apply Proposition~\ref{prop:rect bound}, we require that $\ell-3(n-1)w>0$.   By Equation~\ref{eq:ell bd}, 
  \[\ell -3(n-1)w> 0.5-w -3(n-1)w=0.5 -(3n-2)w,\] so we need that
  \[w< \frac{1}{2(3n-2)}.\]   By Equation~\ref{eq:tan value}, it thus suffices that 
  \[\alpha<\frac{0.98}{3n-2}.\]   
   
   Choosing $\delta_n$ as in Equation~\eqref{eq:delta n} and letting $\alpha\leq \delta_n$, Proposition~\ref{prop:rect bound} and Equations~\ref{eq:ell bd} and \ref{eq:tan value} yield
\begin{equation}\label{eq:mess}\sigma_k(\Omega)\leq \frac{2k^2(n-1)^2\pi^3 w}{(0.5-w-3(n-1)w)^2}< \frac{2k^2(n-1)^2\pi^3 \frac{\alpha}{2(0.98)}}{(\frac{1}{2}-(3n-2)\frac{\delta_n}{2(0.98)})^2}.\end{equation}
\end{proof}

\begin{figure}[h] \centering 
\begin{tikzpicture}
\draw  (0,0) -- (8, -1); 
\draw  (0,0) -- (8, 1); 
\draw (8, -1) -- (8,1); 
\node at (-0.5, 0) {\small $(0,0)$}; 
\node at (1,0){\small $\alpha$}; 
\draw[dashed] (0,1.1) --  (8,1.1);
\draw[dashed] (0,-1.1) --  (8,-1.1);
\draw[dashed] (8,1) --  (8,1.1);
\draw[dashed] (8,-1.1) --  (8,-1);
\draw[dashed] (0,-1.1) --  (0,1.1);
\draw [thick,decoration={brace,mirror},decorate] (0,-1.3) -- (8,-1.3); 
\node at (4, -1.7) {$\ell$}; 
\node at (8.6, 0) {$w$}; 
\draw [thick,decoration={brace,mirror},decorate] (8.3,-1.1) -- (8.3,1.1); 
\end{tikzpicture}
\caption{The polygon $\Omega$ (not shown) lies inside an isosceles triangle, which in turn lies inside a rectangle. One vertex of $\Omega$ is at the origin and at least one vertex of $\Omega$ lies on the righthand edge of the isosceles triangle and thus of the rectangle.}
\label{fig:placeP}
\end{figure}
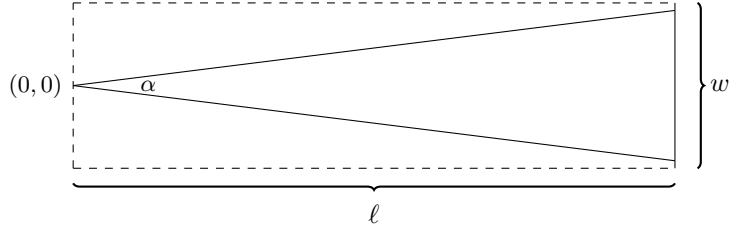

The eigenvalue bounds in Theorem~\ref{th:awesome} can be reversed to yield an inverse spectral result: 

\begin{cor}\label{cor:awesome2}
Given $n$, let $\delta_n$ and $C_n$ be as in Equations~\eqref{eq:delta n} and \eqref{eq:C n}, and let $k$ be any positive integer. Then for all convex $n$-gons $\Om$, the interior angles $\alpha_1,\dots, \alpha_n$ of $\Om$ satisfy 
$$\alpha_j\geq \min\left\{\delta_n, \frac{\sigma_k(\Om)L(\pa \Om)}{C_n k^2}\right\},\,\,j=1,\dots, n.$$
Thus a lower bound on the $k$th perimeter-normalized Steklov eigenvalue yields a lower bound on the angles of $\Om$.   In particular, there exists a uniform lower bound on the angles of any collection of mutually Steklov isospectral convex $n$-gons.
\end{cor}

Theorem~\ref{th:ss_ngons}, Lemma~\ref{lem: |c|}, and Corollary~\ref{cor:awesome2} together imply that the characteristic polynomial of an admissible convex $n$-gon $\Om$ along with a lower bound on the $k$th Steklov eigenvalue for some $k\in\Z^+$ suffice to determine $\Om$ up to finitely many possibilities among all convex $n$-gons.  We will see in the next section that the characteristic polynomial alone suffices to obtain finiteness of Steklov isospectral admissible convex $n$-gons.  However, Corollary~\ref{cor:awesome2} will play a role in extending the spectral finiteness results to a larger class of $n$-gons in Section~\ref{sec: finiteness}.  Using Corollary~\ref{cor:awesome2}, it is possible to obtain finiteness of certain Steklov isospectral sets of convex polygons, but it is not clear if that result alone suffices to obtain an upper bound on the number of such mutually Steklov isospectral non-congruent polygons.  For this reason, in the next section we will use a different approach to obtain explicit bounds on the size of such sets.

\section{Bounds on the sizes of Steklov isospectral sets of admissible convex polygons}\label{sec: isospec size}
We will give upper bounds on the number of mutually non-congruent convex $n$-gons that can be Steklov isospectral to a given admissible convex $n$-gon.  Although we expect the following result is contained in the literature, we include it with a short proof, since it is essential to our results.

\begin{lemma}\label{lem:n-3}
Let $\Omega$ be a convex $n$-gon.  Assume that we know the cyclically ordered side lengths $\bell=(\ell_1,\dots, \ell_n)$ and the corresponding vector of interior angles $\balpha=(\alpha_1,\dots, \alpha_n)$ but with three of the entries replaced by blank place holders.  Then we can uniquely determine the three missing angles and therewith $\Omega$ up to congruence.
\end{lemma}

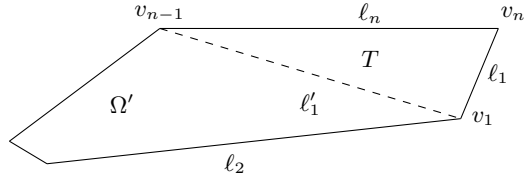
\begin{figure}[h] \centering 
\begin{tikzpicture}
\draw  (-3, -0.5) -- (-1, 1);  
\draw[dashed] (-1,1) -- (3, -0.2);
\draw (-1, 1) -- (3.5, 1); 
\draw  (-3, -0.5) -- (-2.5, -0.8); 
\draw  (-2.5, -0.8) -- (3, -0.2); 
\draw (3, -0.2) -- (3.5, 1); 
\node at (0, -0.8) {\small $\ell_2$}; 
\node at (-1, 1.2) {\small $v_{n-1}$};
\node at (3.7, 1.2) {\small $v_n$};
\node at (1.8, 1.2) {\small $\ell_{n}$};
\node at (3.3, -0.2) {\small $v_1$};
\node at (3.5, 0.4) {\small $\ell_1$};
\node at (-1.5, 0) {\small $\Omega'$}; 
\node at (1.8, 0.6) {\small $T$}; 
\node at (1, 0) {\small $\ell_1'$}; 

\end{tikzpicture}
\caption{A convex $n$-gon (in this case $n=5$) is shown here with $v_n$ a vertex whose interior angle is known.  We divide $\Omega$ by drawing a line segment from $v_{n-1}$ to $v_1$, splitting $\Omega$ into a convex $(n-1)-$gon $\Omega'$ and a triangle $T$.  }
\label{fig:omega'}
\end{figure} 

\begin{proof}
We prove the lemma by induction.  The lemma holds when $n=3$ since triangles that have all their side lengths in common are congruent.   Now let $n>3$ and assume the lemma holds for $(n-1)$-gons.  Let $\Omega$ be an $n$-gon with the given data.  Denote by $v_1,\dots, v_n$ the vertices with the corresponding angles $\alpha_1,\dots, \alpha_n$.  Let
$$\mathcal{K}=\{j\in \{1,\dots, n\}:\alpha_j\mbox{\,is\,known}\}.$$
For notational convenience in what follows, we assume without loss of generality that $n\in \mathcal{K}$.  (Otherwise, we may cyclically permute the entries of $\bell$ and $\balpha$.)  In particular, the edges $v_{n-1}v_n$ and $v_nv_1$ adjacent to $v_n$ have lengths $\ell_n$ and $\ell_1$ respectively.  The line segment $v_{n-1}v_1$ divides $\Omega$ into a triangle $T$ with vertices $v_{n-1}, v_n, v_1$ and a necessarily convex $(n-1)$-gon $\Omega'$ with vertices $v_1, \dots, v_{n-1}$  as in Figure~\ref{fig:omega'}.  Since we know the angle of $T$ at vertex $v_n$ and the lengths of the two adjacent sides, we can determine $T$.  In particular, we can read off the length $\ell_1':=|v_{n-1}v_1|$.  The remaining cyclically ordered side lengths of $\Omega'$ are given by $\ell'_j=\ell_j$, $j=2,\dots, n-1$.   The angle of $\Omega'$ at vertex $v_1$ is the difference between the angles of $\Omega$ and $T$ at that vertex and similarly for the angle at $v_{n-1}$.   
Define $\mathcal{K}'$ analogously to $\mathcal{K}$.    Since the interior angles of $T$ are known, one easily sees that 
$$\mathcal{K}'=\mathcal{K}\cap \{1,\dots, n-1\}$$
and thus $\vert\mathcal{K}'\vert=(n-1)-3$; i.e., the only missing data for $\Omega'$ consists of three angles.  The induction hypothesis yields these three remaining angles of $\Omega'$, and we can determine the three missing angles of $\Omega$.
\end{proof}

\begin{thm}\label{thm:upper bound} Let $\Omega$ be a convex admissible $n$-gon and let $\isop(\Omega)$ be the set of all congruence classes of convex $n$-gons (necessarily admissible) that have the same characteristic polynomial as $\Omega$.   Then the order $ \vert\isop(\Om) \vert$ of $\isop(\Om)$ satisfies the following:
\begin{enumerate}
\item[(a)] If $\Omega$ has no even angles, we have $ \vert\isop(\Omega)\vert \leq \binom{n}{3}.$  
\item[(b)] If $\Omega$ has exactly one even angle, then $ \vert \isop(\Omega) \vert\leq\binom{n-1}{n-3}=\binom{n-1}{2}.$
\item[(c)] If $\Omega$ has exactly two even angles, then $ \vert\isop(\Omega) \vert\leq 4(n-2).$  This bound can be improved to $2(n-2)$ if the even angles are adjacent.
\item[(d)] If $\Omega$ has three even angles, then $ \vert\isop\Omega)\vert \leq 8$.  This bound can be improved to 4 if two of the even angles are adjacent and to 2 if all three of the even angles are in consecutive order.
\end{enumerate}
\end{thm}

\begin{proof}

Recall that the characteristic polynomial determines the number of even angles (see Theorem~\ref{th:ss_ngons}). 
(a) By Theorem~\ref{th:ss_ngons}(b) and the fact that $\Omega$ is admissible, the characteristic polynomial determines $\bell(\Omega)$ and
$\pmca(\Omega)$ modulo a choice of orientation and cyclic labelling.   (See Notation and Remarks~\ref{nota:excep comps} for the definition of $\pmca(\Om)$.)   For every convex $n$-gon $\Om$, at least $n-3$ of the interior angles are obtuse, and Lemma~\ref{lem: |c|} tells us that $|c|$ is injective on the set of all obtuse angles.   Thus, by Lemma~\ref{lem:n-3}, $\Om$ is uniquely determined up to congruence by $\bell(\Om)$, the locations (i.e., the corresponding subscripts $j$) of $n-3$ obtuse angles among the $\alpha_j$'s, and the corresponding values of $|c|(\alpha_j)$ for these obtuse angles.  There are $\binom{n}{n-3}=\binom{n}{3}$ possible ways that the obtuse angles may be distributed among $\alpha_1,\dots, \alpha_n$. 

(b) We may choose the labeling so the unique even angle is $\alpha_n$.   By Remark~\ref{rem: rem on thm}, the characteristic polynomial again determines both $\bell(\Om)$ and $\pmca(\Om)$ up to orientation and cyclic relabeling.  There are $\binom{n-1}{n-3}=\binom{n-1}{2}$ possible ways that $n-3$ obtuse angles may be distributed among $\alpha_1,\dots, \alpha_{n-1}$, and (b) follows.

(c) Let $\alpha_m$ and $\alpha_n$ be the two even angles; here $m\in \{1,\dots, n-1\}$.  The exceptional components then satisfy $\bell(\y_1)=(\ell_1,\dots,\ell_m)$,  $\bell(\y_2)=(\ell_{m+1},\dots, \ell_n)$, $\pmca(\y_1)=(|c|(\alpha_1),\dots, |c|(\alpha_{m-1}))$, and $\pmca(\y_2)=(|c|(\alpha_{m+1}),\dots, |c|(\alpha_{n-1}))$.    Corollary~\ref{cor:excep}   tells us that this information is determined up to the four possible reorderings that arise from the choices of $\y_i$ versus $-\y_i$.   Once the ordering is fixed, it remains to choose $n-3$ obtuse angles among the $n-2$ angles $\{ \alpha_1, \dots, \alpha_{n-1}\} \setminus \{\alpha_m\}$ in order to determine $\Om$.   Thus $\Om$ is spectrally determined up to at most $4\binom{n-2}{n-3}=4(n-2)$ possibilities.  If the even angles are adjacent, then one of the exceptional components $\y_i$ consists of a single edge and $\bell(\y_i)=\bell(-\y_i)$.   Thus we have only two rather than four possible reorderings, proving the final statement in part (c).

(d)  The proof is similar to that of (c).   We now have three exceptional components, each of which may undergo a change of orientation, so we have $2^3=8$ possible reorderings.   Since we have three even, thus non-obtuse, angles, all the remaining angles are obtuse so there are no further choices to be made.  The characterisic polynomial thus determines $\Omega$ up to 8 possibilities.  If two of the even angles are adjacent, then the exceptional component between them consists of a single edge and thus $\bell(\y_i)=\bell(-\y_i)$, so we are reduced to $2^2=4$ possibilities.   If all three even angles are in consecutive order, then two exceptional components are singleton edges and only the orientation of the remaining exceptional component remains to be determined, thus reducing the size of the isospectral set to at most 2.
\end{proof}

Since the characteristic polynomial is a Steklov spectral invariant, our theorem also quantifies the maximum number of congruence classes of convex admissible $n$-gons that have a common Steklov spectrum.  Moreover, for certain convex admissible $n$-gons, that number is one:

\begin{prop} \label{th:finitelymany}~
Let $\Omega$ be a convex admissible $n$-gon all of whose angles are obtuse.  Then $\Omega$ is uniquely determined up to congruence by its Steklov spectrum within the set of all convex $n$-gons.
\end{prop}

\begin{proof} 
The assumption that all angles of $\Omega$ are obtuse says, in particular, that there are no even angles.  
Thus the spectrum determines $\bell(\Omega)$ and
$\pmca(\Omega)$ modulo a choice of orientation and cyclic labelling.  By Lemma~\ref{lem: |c|}(c), the map $|c|: \left(\frac{\pi}{2}, \pi\right) \rightarrow (0,1)$ is one-to-one on the set of obtuse angles.  Consequently, if $\Om'$ is another convex $n$-gon with $\pmca(\Om') =\pmca(\Om)$, then the sum of all the angles of $\Om'$ will be less than $(n-2)\pi$ unless $\balpha(\Om')=\balpha(\Om)$.  Thus  $\Om'$ is congruent to $\Om$. 
\end{proof}

In the proofs of Theorem~\ref{thm:upper bound} and Proposition~\ref{th:finitelymany}, we did not use the full strength of the spectral invariant $\pm\pmc(\Omega)$ since we instead used $\pmca(\Om)$.   We can sometimes improve the upper bound by using the stronger invariant, as we now demonstrate. 

\begin{prop}\label{prop:improved bounds}
Let $\Omega$ be a convex admissible $n$-gon and let $\isos(\Omega)$ be the maximal set of all congruence classes of convex $n$-gons that are Steklov isospectral to $\Omega$.  Denote by $b$ the number of interior angles of $\Omega$ that lie in 
\[ B^+:=\{\alpha\in(0,\pi): \,0<c(\alpha)<1\}=\bigcup_{m\in4\Z^+}\, \left(\frac{\pi}{m+1}, \frac{\pi}{m}\right)\cup  \left(\frac{\pi}{m}, \frac{\pi}{m-1}\right). \] 
If $n\geq 5$, and if $\Om$ has no even angles, then $\vert\isos(\Omega)\vert \leq \binom{n-b}{3}$.  If $n\geq 6$ and if $\Om$ has one even angle, then $\vert\isos(\Omega) \vert\leq \binom{n-1-b}{2-b}$.  This result also holds when $n=5$ provided that $b\leq 1$.  
\end{prop}

\begin{proof}
We first make some general observations.  The fact that all elements of $B^+$ are less than $\frac{\pi}{3}$ implies that $b \leq 2$. Moreover, if $\Omega$ has an even angle less than $\frac{\pi}{2}$, then $b \leq 1$.  If $\Om$ either has two even angles whose sum is less than $\frac{3\pi}{4}$, or if $\Om$ has three even angles, then $b=0$. 

We now assume $n \geq 5$, and $\Om$ has no even angles.  Since $\Omega$, being admissible, has no odd angles, and has no even angles, all the entries of $\pmc(\Om)$ lie in $(-1,0) \cup (0,1)$.  Thus $b$ is precisely the number of positive entries in $\pmc(\Om)$.  Since $n\geq 5$, and $b\leq 2$, the number of negative entries must exceed the number of positive entries and thus knowledge of $\pm\pmc(\Om)$ uniquely determines $\pmc(\Om)$. For any obtuse angle $\alpha_j$, the corresponding entry $c(\alpha_j)$ is negative.   Thus in the proof of Theorem~\ref{thm:upper bound}(a), we may replace $\binom{n}{3} $ by $\binom{n-b}{3}$.

Next we assume that $n \geq 6$, and $\Om$ has one even angle.  Following the notation in the proof of Theorem~\ref{thm:upper bound}(b), we need to count the possible ways $n-3$ obtuse angles may be distributed among $\alpha_1,\dots, \alpha_{n-1}$.  Since $\Om$ has only one even angle, an argument analogous to the preceding case allows us to determine the sign of the spectral invariant $\pm\pmc(\y)$ and then to narrow the candidates down to $n-1-b$, from which we must choose $n-3$.   Thus $\vert\isos(\Omega)\vert \leq \binom{n-1-b}{n-3}=\binom{n-1-b}{2-b}$.
\end{proof}

\section{Spectral finiteness results for some classes of weakly admissible polygons}\label{sec: finiteness}

Recall that admissibility of an $n$-gon $\Om$ with all interior angles in $(0,\pi)$ says both that the edge lengths are incommensurable over $\{-1,0,1\}$ and that there are no odd angles.   In this section we obtain spectral finiteness results for convex $n$-gons satisfying significantly weaker hypotheses.  

\begin{defn}\label{def: omega reduced} 
Let $\Om$ be a convex $n$-gon.
 \begin{enumerate} 
  \item[(a)]
 Let $k$ be the number of odd interior angles in $\Om$.  If $k=0$, set $\omred:=\Om$.  If $k=1$ or 2, let $\omred$ be a curvilinear $(n-k)$-gon obtained by ``removing'' the vertices where the odd angles occur.   More precisely, if $\alpha_j$ is an odd angle and $\ell_j$ and $\ell_{j+1}$ are the lengths of the two edges that meet at the vertex with angle $\alpha_j$, then replace the two edges by a single smooth curve of length $\ell_j+\ell_{j+1}$, being careful not to affect the adjacent vertex angles $\alpha_{j-1}$ and $\alpha_{j+1}$.   If there are two odd angles, repeat the process.   In particular, if odd angles $\alpha_{j-1}$ and $\alpha_j$ occur at adjacent vertices of $\Om$, then the three edges incident to these two vertices are replaced by a single smooth curve of length $\ell_{j-1}+\ell_j+\ell_{j+1}$.  The only convex polygons with more than two odd angles are equilateral triangles.   In this case, $\omred$ is a smooth simply-connected domain, and the characteristic polynomial of $\omred$ is defined as in Remark~\ref{rem: smooth}. We refer to $\omred$ as the reduced curvilinear $(n-k)$-gon associated with $\Om$.   
   
 \item[(b)] We say that a convex $n$-gon is \emph{weakly edge-admissible} if the edge lengths of $\omred$ are incommensurable over $\{-1,0,1\}$. Observe that incommensurability of the edge lengths of $\Om$ over $\{-1,0,1\}$ implies that $\Om$ is weakly edge-admissible.
\end{enumerate}
\end{defn}

\begin{remark}\label{rem:wd}~
We note that $\omred$ is well-defined only up to the choice of the smooth curves replacing the pairs of edges that meet at an odd angle.  In what follows, the choice of curves will not matter.   What will be important are the lengths of these smooth curves and the fact that they are not straight line segments.  The latter distinguishes them from the other edges of $\omred$.  
\end{remark}

 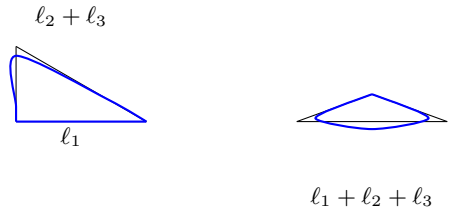
\begin{figure}[h] \centering 
\begin{tikzpicture}
\draw   (-1.732,0) --  (-1.732,1);
\draw  (-1.732, 1) -- (0, 0); 
\draw (0, 0) -- (-1.732, 0); 
\node at (-1, -0.2) {\small $\ell_1$}; 
\node at (-1, 1.4) {\small $\ell_2+\ell_3$}; 
\draw[thick, blue] plot [smooth] coordinates {(-1.732, 0) (-1.732, 0.2) (-1.732, 0.875) (-0.5, 0.2887) (0, 0)};
\draw[thick, blue] (0,0) -- (-1.732,0);

\draw (2,0) -- (4,0); 
\draw (2,0) -- (3, 0.364); 
\draw (3, 0.364) -- (4,0);  

\draw[thick, blue] (3,0.364) -- (2.75,0.273);
\draw[thick, blue] plot [smooth] coordinates {(2.75,0.273) (2.25, 0.04) (3,-0.1) (3.75,0.04) (3.25,0.273)};
\draw[thick, blue] (3.25,0.273) -- (3,0.364);

\node at (3, -1) {\small $\ell_1+\ell_2+\ell_3$}; 

\end{tikzpicture}
\caption{On the left, a $30^\circ \ndash 60^\circ \ndash 90^\circ$ triangle, having one odd angle, is shown together with its associated reduced curvilinear $2$-gon in blue.  On the right, a triangle with two odd angles each measuring $\frac \pi 9$ is shown together with its associated reduced curvilinear $1$-gon in blue.  }
\label{fig:omegared}
\end{figure}

Large classes of polygons are weakly edge-admissible.  In particular, triangles with one odd angle are weakly edge-admissible by the triangle inequality.  If a triangle has $2$ or $3$ odd angles, then the only edge length of $\omred$ is its perimeter; with $3$ odd angles, its reduced curvilinear polygon is a smoothly bounded domain. Examples of triangles with one and two odd angles and their associated reduced curvilinear polygons are shown in Figure \ref{fig:omegared}.  
In addition to triangles, every convex quadrilateral $\Om$ that has two adjacent odd interior angles is necessarily weakly edge-admissible.  Indeed, suppose angles $\alpha_2$ and $\alpha_3$ are odd.    Then $\omred$ has only two edges of lengths $\ell_1':=\ell_1$ and $\ell_2':=\ell_2+\ell_3+\ell_4$, respectively, where the $\ell_j$'s are the edge lengths of $\Om$.  Since all edges have positive length and since we necessarily have $\ell_1<\ell_2+\ell_3+\ell_4$, the set $\{\ell_1',\ell_2'\}$ is incommensurable over $\{-1,0,1\}$.   

\begin{lemma}\label{om vs omred} We use the notation of Definition~\ref{def: omega reduced}.   Let $\Om$ be a weakly edge-admissible convex $n$-gon.  Let $k$ be the number of odd interior angles in $\Om$.   Then:  
\begin{enumerate}
\item[(a)] $\omred$ is either an admissible curvilinear $(n-k)$-gon or a domain with smooth boundary if $n=k=3$;
\item[(b)] 
The characteristic polynomials of $\Om$ and $\omred$ are identical except possibly for a change in the sign of the constant term.  The sign will depend on the parity of the odd angles in the sense of Definition~\ref{def:generic}.
\item[(c)] If $\Om$ is not an equilateral triangle, the characteristic polynomial of $\Om$
determines $\pmca(\omred)$ and $\bell(\omred)$ up to possible permutations of the entries. Moreover, unless $\Om$ has more than one even angle, the characteristic polynomial of $\Om$ determines $\pmca(\omred)$ and $\bell(\omred)$ uniquely (modulo the choice of boundary orientation and cyclic labelling).  

\end{enumerate}
\end{lemma}

We have excluded equilateral triangles in part (c) only because we have not defined $\pmca(\omred)$ when $\omred$ has smooth boundary.

\begin{proof}
(a) is immediate from Definitions~\ref{def:generic} and \ref{def: omega reduced}.   

(b) Under the hypothesis of weak edge-admissibility, it is straightforward to see from Definition~\ref{def: char poly} that the non-constant terms of the characteristic polynomials of $\Om$ and $\omred$ are identical, since $c(\alpha_j)=0$ when $\alpha_j$ is odd.   Moreover weak edge-admissibility implies that the constant term in the characteristic polynomial of $\Om$ is given by $\prod_{j=1}^n\,\sin\left(\frac{\pi^2}{2\alpha_j}\right).$  Any odd angles contribute a factor of $\pm 1$ to this product, while the product of the remaining factors yields the constant term in the characteristic polynomial of $\omred$.   

(c) Observe that for any admissible curvilinear polygon $\Sigma$, the data $\pmc(\Sigma)$ and $\bell(\Sigma)$ are independent of the sign of the constant term in the characteristic polynomial of $\Sigma$.   We can now apply parts (a) and (b) along with Theorem~\ref{th:ss_ngons} and Remark~\ref{rem: rem on thm} to complete the proof.
\end{proof}

Before addressing spectral finiteness, we observe the following consequence of Lemma~\ref{om vs omred}:

\begin{prop}\label{prop:rational}
Let $\Om_1$ and $\Om_2$ be weakly edge-admissible convex $n$-gons that have the same characteristic polynomial. If all angles of $\Om_1$ are rational multiples of $\pi$, then the same is true for all angles of $\Om_2$.
    \end{prop}

    \begin{proof}  Applying Remark~\ref{rem:rat} along with Lemma~\ref{om vs omred}(c), we see that all angles of $\omred_2$ are rational multiples of $\pi$.   The only remaining angles of $\omred_2$ are odd angles, which are necessarily rational multiples of $\pi$.
    \end{proof}
    
\begin{thm}\label{thm: not two adjacent odd}  Let $\cp$ be the set of all weakly edge-admissible convex polygons; moreover, assume that if the polygon contains two odd angles, then they are adjacent.  Let $S$ be any subset of $\cp$ consisting of congruence classes of convex polygons that have the same characteristic polynomial and that share a common lower bound on their $k$th perimeter-normalized Steklov eigenvalue for some $k\in \Z^+$.  Then $S$ is finite.  In particular, any set of mutually Steklov isospectral elements of $\cp$ is finite.
\end{thm}

\begin{figure}[h] \centering 
\begin{tikzpicture}
\draw  (-3, 0) -- (1, 0);
\draw (-3, 0) -- (-5, 2); 
\draw[dashed] (-5, 2) -- (-6, 3); 
\node at (-6, 2.5) {$R$}; 
\draw (1, 0) -- (3, 1); 
\draw (-5, 2) -- (3, 1); 
\draw[dashed] (3, 1) -- (4.4, 1.7);
\draw[dashed] (-6,3) -- (4.4,1.7);

\node at (-2.8, 0.2) {\small $\theta_1$}; 
\node at (-3, -0.2) {\small{$(0,0)$}};
\node at (0.7, 0.2) {\small $\theta_2$};
\node at (0.9, -0.2){\small $(\ell, 0)$}; 
\node at (-1, -0.2) {\small{$I$}}; 

\node at (4.3, 1.2) {$R'$};
\node at (-4.3, 1.7) {\small{$\theta_4$}};
\node at (2.1, 0.8) {\small{$\theta_3$}}; 

\end{tikzpicture}
\caption{A convex quadrilateral is shown here with fixed interior angles. Connecting points on the rays $R$ and $R'$ by lines parallel to the side of the quadrilateral connecting the upper vertices generates a family of quadrilaterals with the same interior angles.}
\label{fig:quad_parallel}
\end{figure}
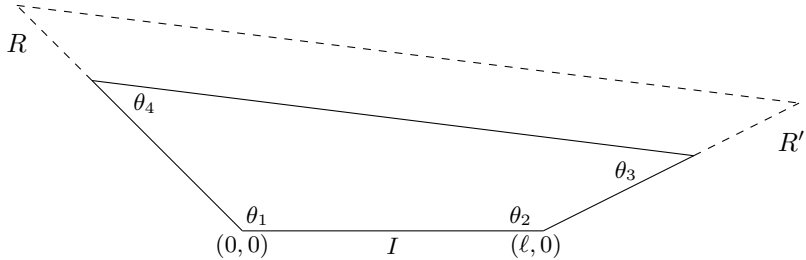

To prove Theorem \ref{thm: not two adjacent odd} we require the following geometric lemma.
\begin{lemma}\label{lem:Euc geom}
A convex quadrilateral is uniquely determined up to congruence by its four labeled angles, one labeled side length, and its perimeter.
\end{lemma}

\begin{proof}[Proof of Lemma~\ref{lem:Euc geom}]
Let $Q$ be a convex quadrilateral with the given data.  Let $\ell$ be the known side length.  Situate $Q$ in the plane so that the edge with the prescribed length is the interval $I$ on the $x$-axis with endpoints $(0,0)$ and $(\ell, 0)$ and such that $Q$ lies in the closed upper half plane.  There are two edges adjacent to $I$ on the rays $R$ and $R'$ emanating upwards from the endpoints of $I$ at the prescribed angles; the fourth edge of $Q$, which is opposite $I$, must have endpoints on $R$ and $R'$ and make the prescribed angles with these rays.    The assumption that there exists at least one quadrilateral with the given data  guarantees the existence of at least one line segment joining $R$ and $R'$ at the prescribed angles.  Then there exists a continuum of such segments, all mutually parallel as in Figure \ref{fig:quad_parallel}.   Each gives rise to a convex quadrilateral with the prescribed angles and side length.   However, the perimeters of these quadrilaterals strictly increase as the distance from the segment to the $x$-axis increases.  Thus there can be only one such quadrilateral with the prescribed perimeter.    
\end{proof}

\begin{proof}[Proof of Theorem~\ref{thm: not two adjacent odd}]
Equilateral triangles are distinguishable from other elements of $P^{*}$ by the number of cosine terms in their characteristic polynomials (see Lemma~\ref{om vs omred}, Remark~\ref{rem: smooth}, and the observations immediately preceding Proposition~\ref{prop:vertices}). Thus for notational simplicity, we will exclude equilateral triangles in the remainder of the proof. 

Write $$P^{*} =\bigcup_{n=3}^\infty\,\cpn,$$ where $\cpn$ consists of all convex $n$-gons in $\cp$.   For $\Om\in\cp$,  Lemma~\ref{om vs omred} implies that the characteristic polynomial of $\Om\in \cp$ determines the number of vertices in $\omred$.  Since each $\Om\in \cp$ has at most three more vertices than $\omred$, any set $S$ as above can intersect $\cpn$ for at most four values of $n$. To prove finiteness, it thus suffices to fix $n$ and show that each $\Om\in\cpn$ is determined up to finitely many possibilities in $\cpn$ by its characteristic polynomial and a Steklov eigenvalue bound as in the statement of the theorem.

   If $\Om$ has no odd angles, then it is necessarily admissible and we may apply Theorem~\ref{thm:upper bound} to complete the proof.   Thus we assume that $\Om$ has at least one odd angle.  Each of the following are determined up to finitely many possibilities by the characteristic polynomial and the eigenvalue bound:
   
   \begin{enumerate}
   \item[(i)] $\pmca(\omred)$ and $\bell(\omred)$ by Lemma~\ref{om vs omred};
     \item[(ii)] the number of odd angles in $\Om$, since $n$ is fixed and we know the number of angles in $\omred$;
   \item[(iii)] $\balpha(\omred)$ and also the values of the odd angles by (i), Corollary~\ref{cor:awesome2} and Lemma~\ref{lem: |c|};

   \item[(iv)] the location of the odd angles: indeed, all but one of the edges of $\omred$ is a straight line segment, since all odd angles of $\Om$ are assumed to be adjacent.  There are only finitely many choices for this edge and thus for the odd angles.
   \end{enumerate}

To complete the proof of finiteness, it thus suffices to fix a choice of the data (i)--(iv) and show that there is at most one convex $n$-gon with the given data.   The data gives us  \emph{all} the angles $\alpha_1,\dots,\alpha_n$ of $\Om$, the lengths of all the edges that join the non-odd angles, and the sum of the lengths of those edges that are adjacent to odd angles (in particular, the perimeter of $\Om$).

If $n=3$, the angles along with the perimeter determine $\Om$.    Thus we assume $n\geq 4$.   Let $k\in \{1,2\}$ be the number of odd angles of $\Om$.   For notational simplicity, we cyclically relabel the vertices of $\Om$ so that $\alpha_n$, and also $\alpha_{n-1}$ if $k=2$, are the odd angles.   In addition to knowing all the angles of $\Om$, we know $\ell_2,\dots, \ell_{n-k}$ and the perimeter.   
It remains to determine the remaining lengths.

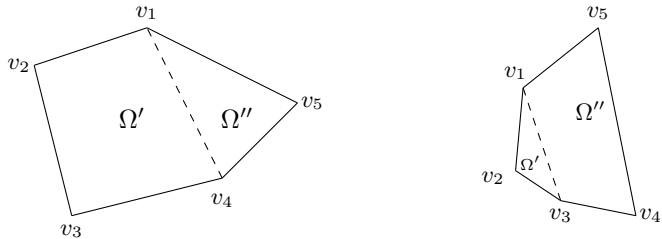
\begin{figure}[h] \centering 
\begin{tikzpicture}
\draw  (-3, 0) -- (-1, 0.5);
\draw (-1, 0.5) -- (0, 1.5); 
\draw (0, 1.5) -- (-2, 2.5); 
\draw (-2, 2.5) -- (-3.5, 2); 
\draw (-3.5, 2) -- (-3, 0); 
\draw[dashed] (-1, 0.5) -- (-2, 2.5);

\node at (-3, -0.2) {\small $v_3$}; 
\node at (-1, 0.2) {\small{$v_4$}};
\node at (0.2, 1.5) {\small $v_5$};
\node at (-2, 2.7){\small $v_1$}; 
\node at (-3.7, 2) {\small{$v_2$}}; 
\node at (-2.2, 1.3) {$\Omega'$};
\node at (-0.8, 1.3) {$\Omega''$}; 

\draw (3.5, 0.2) -- (4.5, 0); 
\draw (3.5, 0.2) -- (2.9, 0.6); 
\draw (2.9, 0.6) -- (3, 1.7); 
\draw (3, 1.7) -- (4, 2.5);  
\draw (4, 2.5) -- (4.5, 0); 
\draw [dashed] (3, 1.7) -- (3.5, 0.2); 
\node at (3.1, 0.7) {\tiny $\Omega'$}; 
\node at (3.9, 1.4) {\small $\Omega''$}; 
\node at (3.5, 0) {\small $v_3$};
\node at (4.7, 0) {\small $v_4$}; 
\node at (4, 2.7) {\small $v_5$}; 
\node at (2.6, 0.5) {\small $v_2$}; 
\node at (2.9, 1.9) {\small $v_1$};

\end{tikzpicture}
\caption{Two convex 5-gons are shown here, the left having one odd angle at vertex $v_5$ while the right has two odd angles at vertices $v_4$ and $v_5$.  The 5-gons are split into $\Omega''$ and $\Omega'$ by joining $v_1$ to $v_4$ or $v_3$, respectively. 
}
\label{fig:3versions}
\end{figure} 

If $n=4$, then we can apply Lemma~\ref{lem:Euc geom}, with $\ell_2$ playing the role of the known edge length, to complete the proof. 
Thus assume $n\geq 5$.   For $j=1,\dots, n$, we denote by $v_j$ the $j$th vertex of $\Om$ and by $e_j$ the $j$th edge, so the angle at $v_j$ is $\alpha_j$ and the length of $e_j$ is $\ell_j$.  
The line segment $v_{n-k}v_1$ divides $\Om$ into a convex $(n-k)$-gon $\Om'$ with vertices $v_1,\dots, v_{n-k}$ and edges $e_2,\dots, e_{n-k}, v_{n-k}v_1$ and a convex $(k+2)$-gon $\Om''$ (so a triangle or a quadrilateral) whose edge set consists of $v_{n-k}v_1$ along with the edges of $\Om$ adjacent to the odd angle(s).  This is shown in Figure \ref{fig:3versions}.  It's easy to see that the known data $\ell_2,\dots, \ell_{n-k}, \alpha_2,\dots, \alpha_{n-k-1}$ determines $\Om'$.   Consequently, we know the edge length $|v_{n-k}v_1|$ and we know the angles of $\Om'$ at $v_{n-k}$ and $v_1$.   From these angles along with our knowledge of $\alpha_{n-k}$ and $\alpha_1$ from $\balpha(\omred)$, we also know the angles of $\Om''$ at these two vertices.  Thus we know the information ASA (angle-side-angle) for $\Om''$.  If $k=1$ so that $\Om''$ is a triangle, this determines $\Om''$.    If $k=2$, we use Lemma~\ref{lem:Euc geom} and the fact that we also know the sum of the remaining edge lengths of $\Om''$ (equivalently, we know the perimeter of $\Om''$) to recover $\Om''$.    We have thus determined the remaining side lengths of $\Om$, completing the proof.
\end{proof}

To complete our discussion of weakly edge-admissible convex polygons it remains to consider those with two non-adjacent odd angles.  Note that any such polygon necessarily has at least four vertices.   We will denote by $\cpp$ the class of all such convex polygons and write 
$$\cpp=\cup_{n=4}^\infty\, \cpp(n)$$
where $\cpp(n)$ consists of all $n$-gons in $\cpp$.   

 The known Steklov spectral invariants do not suffice to show in full generality that Steklov isospectral sets of such polygons are finite.    Indeed, we will see below that some convex polygons in this class can be continuously deformed while keeping all angles fixed and keeping the characteristic polynomial fixed.  However, we will also show that most convex polygons in this class are finitely determined within $\cpp$ by their characteristic polynomials alone.    
  
 \begin{prop}\label{prop: non adjacent odd}
 For $\Om \in \cpp$, the characteristic polynomial yields the following data:
 \begin{enumerate}
 \item[(a)] the number $n$ of vertices;
 \item[(b)] $\pmc(\omred)$ and $\bell(\omred)$ (both uniquely, modulo the choice of boundary orientation and cyclic labeling);
 \item[(c)] $\balpha(\omred)$ up to at most $n-1$ explicit possibilities and typically uniquely (modulo the choice of boundary orientation and cyclic labeling);   
 \item[(d)] the values of the two odd angles of $\Om$ up to finitely many explicit possibilities.
 \end{enumerate}
Consequently, the characteristic polynomial of $\Om$ determines $\balpha(\Om)$ up to finitely many explicit possibilities.   For each such choice of $\balpha(\Om)$, the characteristic polynomial uniquely determines all the edge lengths except for the pairs of edges incident on the odd angles.   For the latter, the characteristic polynomial determines the sum of the lengths of the edges in each pair.
 \end{prop}

 \begin{proof}
 
(a) follows from Lemma~\ref{om vs omred} since $\Om$ has exactly two more vertices than $\omred$.  Next consider (b).  Since the sum of the two odd angles is at most $\frac{2\pi}{3}$, we have 
\beq\label{eq:pmcomred} \balpha(\omred)\in \left(\frac{\pi}{3}, \pi\right)^{n-2}\eeq 
and at most one angle of $\omred$ is non-obtuse.

In particular, $\omred$ has at most one even angle, and then the even angle must be a right angle.   (b) now follows from Lemma~\ref{om vs omred}(c) along with the fact that $c(\alpha)<0$ for all $\alpha\in (\frac{\pi}{3},\pi)$.   

We apply (b) and Lemma~\ref{lem: |c|} to prove (c).   If $\pmc(\omred)$ has an entry $-1$, necessarily corresponding to a right angle, then $\pmc(\omred)$ determines $\balpha(\omred)$ uniquely since the remaining angles are obtuse. Otherwise, each of the $n-2$ entries in $\pmc(\omred)$ corresponds to a possible location of one non-obtuse angle in $(\frac{\pi}{3}, \frac{\pi}{2})$; for each entry, Lemma~\ref{lem: |c|} (c) implies that we know the angle.  It is also possible that all angles are obtuse, giving a total of $n-1$ possibilities for $\balpha(\omred)$.  To prove generic uniqueness, suppose that $\bgamma:=(x_1\pi,\dots, x_{n-2}\pi)$ and $\bdelta:=(y_1\pi,\dots, y_{n-2}\pi)$ are two of the possible $n-1$ candidates for $\balpha(\omred)$.    If $x_j\neq y_j$, Equation~\eqref{eq:pmcomred} and Lemma~\ref{lem: |c|} imply that one of $x_j\pi,y_j\pi$ lies in $(\frac{\pi}{3},\frac{\pi}{2})$ and the other in $(\frac{\pi}{2},\pi)$.  Since (b) says that  $c(x_j\pi)=c(y_j\pi)$, we then have
$\frac{\pi^2}{2 x_j\pi}=2\pi-\frac{\pi^2}{2 y_j\pi}$.  Thus
$$y_j=\frac{x_j}{4x_j-1} \mbox{\,\,and\,\,}x_j=\frac{y_j}{4y_j-1}.$$
Hence 
\beq\label{eq:yj-xj} y_j\pi-x_j\pi = \pi\frac{2x_j-4x_j^2}{4x_j-1}.\eeq

Define $D$ to be the discrete set given by $D=\{\frac{\pi}{2p+1}+\frac{\pi}{2q+1}:\,\,p,q\in\Z^+\}$.  Let $s(\bgamma)=x_1\pi+\dots +x_{n-2}\pi$ and  $s(\bdelta)=y_1\pi+\dots +y_{n-2}\pi$.  Observe that both $(n-2)\pi-s(\bgamma)$ and $(n-2)\pi-s(\bdelta)$ lie in $D$.    Thus \beq\label{eq:big condition}s(\bgamma)-s(\bdelta)\in D -D =\{a-b: \,a,b\in D\}.\eeq   
Equations~\eqref{eq:yj-xj} and \eqref{eq:big condition} together imply the generic uniqueness of $\omred$.

Next consider (d).   Given any fixed choice of $\balpha(\omred)$ in (c), let $\mu$ be the sum of the entries.   Then the sum of the two odd angles is $(n-2)\pi -\mu$, so at least one of the odd angles is greater than or equal to $\frac{1}{2}[(n-2)\pi -\mu]$.  Hence there are only finitely many possible values for the odd angles, and they are explicitly computable.   

For the final statement of the proposition, items (c) and (d) together yield $\balpha(\Om)$ up to finitely many possibilities.   (Missing from (c) and (d) is the location of the two odd angles -- equivalently the determination of which edges of $\omred$ have non-trivial curvature -- but there are only finitely many possible locations.)   For each of the finitely many choices of $\balpha(\Om)$, the assertion concerning the edge lengths is equivalent to the knowledge of $\bell(\omred)$, guaranteed by (b).
\end{proof}

Given $\Om\in\cpp$, consider the set of all convex polygons in $\cpp$ that have the same characteristic polynomial as $\Om$.    To determine whether this set is finite, it remains only to determine for each of the finitely many choices of $\balpha(\Om)$ in Proposition~\ref{prop: non adjacent odd} whether we can recover the lengths of the edges adjacent to the odd angles from our knowledge of $\balpha(\Om)$ and of the other edge lengths.  The following purely geometric lemma tells us that generically these lengths are uniquely determined but that, when the genericity condition fails, the edge lengths can be continuously deformed without affecting the characteristic polynomial.  For notational simplicity in the lemma, we cyclically relabel the vertices so that the odd angles are labeled $\alpha_1$ and $\alpha_m$ for some $m$.   The restriction on $m$ in the lemma is the condition that the two odd angles are not adjacent.

 \begin{lemma}\label{lem:vectors}  Fix $m$ with $3\leq m\leq n-1$. Suppose that the following data for a convex $n$-gon $\Om$ is known:
 \beq\label{eq: alpha data}\alpha_1, \dots, \alpha_n\eeq and
 \beq\label{eq: weak edge data}\ell_1+\ell_2, \,\ell_3,\dots, \ell_{m-1}, \,\ell_m+\ell_{m+1},\,\ell_{m+2},\dots, \ell_{n}.\eeq

Let $$\Psi=\sum_{i=2}^{m-1}\,(\pi -\alpha_i)\mbox{\,\,and\,\,}\Phi=\sum_{j=m+1}^{n}\,(\pi -\alpha_j).$$
 \begin{enumerate}
  \item[(a)] If $\Psi\neq \Phi$, then $\Om$ is uniquely determined up to congruence by this data.   
   \item[(b)] If $\Psi=\Phi$, then one can continuously deform $\Om$ without changing the data above.
 \end{enumerate}
 \end{lemma}
 
 \begin{proof}
Denote the vertices of $\Om$ by $v_1,\dots, v_n$. We first claim that $\Psi=\Phi$ if and only if the bisector $L_m$ of the angle $\alpha_m$ at $v_m$ is parallel to the bisector $L_1$ of $\alpha_1$ at $v_1$.
  
Consider two polygonal paths between $v_1$ and $v_m$ given by $P: v_1,v_2,\dots, v_{m}$ and $Q: v_1, v_n, \dots, v_{m+1}, v_m$. Since $\Om$ is convex, each of these paths has curvature of constant sign.  Due to the opposite orientations, the curvatures of $P$ and $Q$ have opposite sign.  $\Psi$ and $\Phi$ are precisely the absolute values of the total curvatures of $P$ and $Q$, respectively, and measure the change in the direction of the tangents to the initial and final segments. The initial segments of the two paths make the same angle with $L_1$, differing only by reflection across $L_1$. Consequently, letting $L$ denote the line through $v_m$ parallel to $L_1$, we have $\Psi=\Phi$ if and only if the final segments of $P$ and $Q$ make equal angles with $L$, i.e., if and only if $L_m=L$.  

  Equations~\eqref{eq:anglesum} and \eqref{eq:m'm''}, together with the fact that $\alpha_m=\alpha_m'+\alpha_m''$, yield

 We now prove statements (a) and (b).   We may assume $\Om$ has perimeter one.  Write $h=\ell_1+\ell_2$ and $k=\ell_m+\ell_{m+1}.$
 Situate $\Om$ in the plane and let $\ub_1, \dots, \ub_n$ be unit vectors parallel to the edges $e_1, \dots, e_n$, oriented so that $\ub_j$ points in the direction from $v_{j-1}$ to $v_j$.  Observe that $\ub_2, \dots, \ub_n$ are uniquely determined by $\ub_1$ and the angles $\alpha_1,\dots, \alpha_n$.    Using the fact that the boundary of $\Om$ is a closed polygonal path, we see that the edge lengths satisfy the following system of linear equations:
\begin{equation}\label{eq:vectors}
\begin{cases}\ell_1\ub_1+ \ell_2\ub_2 +\ell_{m}\ub_m+\ell_{m+1}\ub_{m+1}  = -\mathbf{c}\\
\ell_1+\ell_2=h\\
\ell_m+\ell_{m+1}=k
\end{cases}
\end{equation}
where $\mathbf{c}$ is the constant vector
$$\mathbf{c}=\sum_{j\neq 1, 2, m, m+1}\,\ell_j\ub_j.$$
We know that this system has a solution with all the $\ell_j$ strictly positive since we began with the data for a convex $n$-gon $\Om$.   In view of the last two equations, any other solution  $\ell_1', \ell_2', \ell_m', \ell_{m+1}'$ must satisfy 
$$\ell_1'=\ell_1+x,\,\, \ell_2'=\ell_2-x,\,\, \ell_m'=\ell_m+y,\,\,\ell_{m+1}'=\ell_{m+1}-y$$
for some $x, y$.   The first equation then implies that 
\begin{equation}\label{eq: aubu} x(\ub_1-\ub_2) = y(\ub_{m+1}-\ub_m).\end{equation} 

Unless $\ub_2-\ub_1$ is parallel to $\ub_{m+1}-\ub_m$, Equation~\ref{eq: aubu} implies that $x=y=0$, and thus the system given by~\eqref{eq:vectors} has a unique solution; equivalently,  $\Om$ is uniquely determined up to congruence.  Now observe that $\ub_2-\ub_1$, respectively $\ub_{m+1}-\ub_m$, is the bisector of angle $\alpha_1$, respectively $\alpha_m$, in $\Om$.   As noted above, the hypothesis of part (a) is precisely the condition that the two bisectors are not parallel.  This proves (a).

On the other hand, if the bisectors are parallel, i.e., if the hypothesis of part (b) holds, then for any $x$ and $y$ sufficiently small, we get another solution $\ell_1', \ell_2', \ell_m', \ell_{m+1}'$ with all entries positive.   This yields a new closed polygonal path.   By continuity, if $x$ and $y$ are sufficiently small, this path must also bound a convex $n$-gon.   Part (b) now follows.    
 \end{proof}

In the case of quadrilaterals, we can say much more; the following lemma is independent of whether the odd angles are adjacent.

\begin{lemma}\label{lem:remaining angles} Within the class of all weakly edge-admissible convex quadrilaterals with two odd angles, the characteristic polynomial determines whether the two non-odd angles are equal.  Moreover, if they are equal, then the characteristic polynomial determines their value.
\end{lemma}

\begin{proof} 
Let $\Om$ and $\Om'$ lie in this class of quadrilaterals; assume that they have the same characteristic polynomial. Denote by $\gamma$ and $\delta$, respectively $\gamma'$ and $\delta'$, the two angles of $\Om$, respectively $\Om'$, that are not odd. Suppose that one of $\Omega$ and $\Omega'$, say $\Omega$, has two equal angles, i.e., $\gamma=\delta$.  We need to show that $\gamma'=\delta'=\gamma$.  Since the sum of any two odd angles is at most $\frac{2\pi}{3}$, the two equal angles $\gamma$ and $\delta$ are obtuse.  Moreover, 
\beq\label{eq:g'd'} \gamma',\,\delta'\in \left(\frac{\pi}{3}, \pi\right),\eeq and at least one of these angles, say $\gamma'$, is obtuse.    

By Lemma~\ref{om vs omred}(c), we have 
$\{|c|(\gamma'), |c|(\delta')\}= \{|c|(\gamma), |c|(\delta)\}$.  Since $\gamma$ and $\gamma'$ are both obtuse, we must then have $\gamma=\gamma'$ by Lemma~\ref{lem: |c|}(c).    If $\delta'$ is also obtuse, then $\delta'=\gamma=\gamma'$ and we are done.   

Suppose that $\delta'$ is not obtuse.  Then $\delta'\in \left(\frac{\pi}{3}, \frac{\pi}{2}\right)$  by Equation~\ref{eq:g'd'}.  (We can't have $\delta'=\frac{\pi}{2}$ since $|c|(\delta')=|c|(\delta).$) Let $\pi x$ be the sum of the two odd angles of $\Om$, where $x \in (0, \frac{2}{3}]$.   Then the sum of the odd angles of $\Om'$ is given by 
\beq\label{eq:odd sum}S(x):=\pi x + \delta -\delta'.\eeq
To get a contradiction, it suffices to show that $S(x)>\frac{2\pi}{3}$.   
We have $$\delta=\gamma=\pi\frac{2-x}{2}\mbox{\,\,so\,\,}\frac{\pi^2}{2\delta}=\frac{\pi}{2-x}.$$
Since $|c|(\delta)=|c|(\delta')$ and $\frac{\pi^2}{2\delta}\in \left(\frac{\pi}{2},\pi\right)$ while $\frac{\pi^2}{2\delta'}\in \left(\pi,\frac{3\pi}{2}\right)$, we have $$ \frac{\pi^2}{2\delta'}=2\pi - \frac{\pi^2}{2\delta}= \pi\frac{3-2x}{2-x} \mbox{\,\,so\,\,} \delta'= \pi\frac{2-x}{6-4x}.$$

Thus 
$$S(x)=\pi\left[\frac{2-x^2}{3-2x}\right].$$
One easily checks that $S(x)>\frac{2\pi}{3}$, and the lemma follows.
\end{proof}

\begin{remark}\label{rem:deform}
If a weakly edge-admissible convex quadrilateral $\Om$ has two non-adjacent odd angles and two equal non-odd angles, then Lemma~\ref{lem:vectors} shows that one can continuously deform the edge lengths without affecting either the characteristic polynomial or the angles. 
    \end{remark}

    \begin{thm}  
     Let $\Om$ be any weakly edge-admissible convex quadrilateral other than those in Remark \ref{rem:deform}.  
    Then the characteristic polynomial of $\Om$ and a lower bound on the $k$th Steklov eigenvalue for some $k\in\Z^+$ together determine $\Om$ up to finitely many possibilities within the class of all weakly edge-admissible convex quadrilaterals.
\end{thm} 

\begin{proof} 
If $\Omega$ has no odd angles, then it is admissible, and the proof is completed by Theorem~\ref{thm:upper bound}.  If $\Omega$ has one odd angle or has two adjacent odd angles, then the proof is completed by Theorem~\ref{thm: not two adjacent odd}.  We are left with the case that $\Omega$ has two non-adjacent odd angles and two unequal non-odd angles.  Any other weakly edge-admissible quadrilateral with the same characteristic polynomial as $\Omega$ must also have two odd angles, and Lemma~\ref{lem:remaining angles} tells us that the non-odd angles are unequal. Choosing the cyclic ordering so that $\alpha_1$ and $\alpha_3$ are the odd angles, Proposition~\ref{prop: non adjacent odd} yields $\balpha(\Om)$ up to finitely many possibilities.   Lemma~\ref{om vs omred} also yields $\ell_1+\ell_2$ and $\ell_3+\ell_4$.  Thus all the hypotheses of Lemma~\ref{lem:vectors} hold which, together with the fact that $\Phi\neq \Psi$, concludes our proof. 
\end{proof}

\section{Outlook} \label{s:outlook}

We obtained a collection of inverse spectral results for the Steklov eigenvalue problem on polygonal domains.  It is natural to compare these results and, more generally, the results of \cite{klpps21} for simply connected curvilinear domains, with the analogous inverse results for the Laplace eigenvalue problem with Dirichlet or Neumann boundary conditions.  The Laplace spectrum distinguishes simply-connected curvilinear polygons from all bounded plane domains, simply-connected or otherwise, with smooth boundary; see \cite{corners1} (Dirichlet case) and \cite{nrs1} (Neumann and Dirichlet).  The latter article also obtains similar results with Robin boundary conditions.   The article \cite{nrs2} extends these results to more general surfaces with piecewise smooth boundary under an additional hypothesis on the Euler characteristic.  However, it is not known whether the Laplace spectrum detects the number of vertices in a curvilinear domain.    In contrast, while the question of whether the Steklov spectrum can always distinguish curvilinear polygons from smooth domains remains open, the results of \cite{klpps21} for the Steklov problem provide much greater information (e.g., number of vertices, edge lengths) for admissible -- thus generic -- curvilinear $n$-gons with all angles in $(0,\pi)$.   

 To our knowledge, the question of generic finiteness of Dirichlet or Neumann isospectral sets of $n$-gons -- convex or otherwise -- remains open.   This situation contrasts with the results of Section~\ref{sec: isospec size} for the Steklov eigenvalue problem.   However, spectral uniqueness for the Laplace spectrum is known within certain classes of polygons; e.g., triangles are mutually distinguishable by their Laplace spectrum \cite{durso}, \cite{gm_triangles}, with either Dirichlet or Neumann boundary condition. Non-obtuse trapezoids are mutually distinguishable by their Dirichlet spectrum \cite{dtraps} and also by their Neumann spectrum \cite{ntraps}.   The currently known Steklov spectral invariants are not sufficient to mutually distinguish all triangles, although we will see in an upcoming paper that Steklov isospectral sets of triangles are always finite and generic triangles are uniquely determined by their Steklov spectra.   We will also address additional classes of convex $n$-gons.  

Many examples exist, beginning with \cite{gww}, of non-congruent polygonal domains that are isospectral for the Laplacian with both Dirichlet and Neumann boundary conditions.  The maximal possible size of mutually Laplace isospectral sets of non-congruent polygonal domains in the plane is unknown.    
The question of existence of Steklov isospectral plane domains remains open in both the convex and the non-convex case.   However, the known examples of Laplace isospectral plane domains are also isospectral for both a mixed Steklov-Neumann problem and a mixed Steklov-Dirichlet problem \cite{ghw}. 

Although the polygonal examples of Dirichlet and Neumann isospectral plane domains provide a negative answer to Mark Kac's question about hearing the shape of a drum \cite{kac1966}, the question remains open for domains with smooth boundary and for convex domains.    
Watanabe \cite{watanabe2} used heat trace methods to show that there exist oval-shaped domains that are uniquely determined by their Dirichlet (or Neumann) spectra among all bounded planar domains.  Around the same time, Zelditch used wave trace methods to prove that domains with an analytic bi-axisymmetric boundary are uniquely determined by their Laplace spectra within this class of domains \cite{zelditch0}. More recently, Hezari and Zelditch proved that within the class of ellipses with small eccentricity, each element is uniquely determined by its Dirichlet or Neumann Laplace spectrum \cite{2022_ellipses}.  They also proved a similar result for generic real analytic centrally symmetric plane domains \cite{2023_analytic}.   We refer interested readers to the surveys \cite{pre_survey}, \cite{gir_pol},  \cite{112survey} for further reading on the Laplace and Steklov inverse spectral problems.

\end{document}